\RequirePackage{fixltx2e}
\documentclass[10pt]{article}
\usepackage{amsmath}
\usepackage{amsfonts}
\usepackage{amssymb}
\usepackage{amsthm}
\usepackage{enumerate}
\usepackage{verbatim}
\usepackage[capitalise]{cleveref}
\usepackage{graphicx}
\usepackage{microtype}
\usepackage{xifthen}
\usepackage[totalwidth=450pt, totalheight=650pt]{geometry}

\crefname{thm}{Theorem}{Theorems}
\Crefname{thm}{Theorem}{Theorems}
\crefname{conj}{Conjecture}{Conjectures}
\Crefname{conj}{Conjecture}{Conjectures}
\crefname{prop}{Proposition}{Propositions}
\Crefname{prop}{Proposition}{Propositions}
\crefname{cor}{Corollary}{Corollaries}
\Crefname{cor}{Corollary}{Corollaries}
\crefname{defn}{Definition}{Definitions}
\Crefname{defn}{Definition}{Definitions}
\crefname{rmk}{Remark}{Remarks}
\Crefname{rmk}{Remark}{Remarks}
\crefname{prob}{Problem}{Problems}
\Crefname{prob}{Problem}{Problems}

\crefname{figure}{Figure}{Figures}
\Crefname{figure}{Figure}{Figures}

\DeclareMathOperator*{\convts}{\tilde{*}}

\begin{document}

\newcommand{\C}{{\mathbb{C}}}
\newcommand{\R}{{\mathbb{R}}}
\newcommand{\Q}{{\mathbb{Q}}}
\newcommand{\Z}{{\mathbb{Z}}}
\newcommand{\N}{{\mathbb{N}}}
\newcommand{\M}{{\mathbb{M}}}
\newcommand{\sM}{\mathcal{M}}
\newcommand{\sN}{\mathcal{N}}
\newcommand{\sB}{\mathcal{B}}
\newcommand{\sG}{\mathcal{G}}
\newcommand{\sF}{\mathcal{F}}
\newcommand{\sT}{\mathcal{T}}
\newcommand{\sL}{\mathcal{L}}
\newcommand{\sGt}{\widetilde{\sG}}
\newcommand{\sI}{\mathcal{I}}
\newcommand{\sMt}{\widetilde{\sM}}
\newcommand{\Gt}{\widetilde{G}}
\newcommand{\Gh}{\widehat{G}}
\newcommand{\ft}{\tilde{f}}
\newcommand{\lambdat}{\tilde{\lambda}}
\newcommand{\At}{\tilde{A}}
\newcommand{\Pt}{\widetilde{P}}
\newcommand{\Lt}{\widetilde{L}}
\newcommand{\gt}{\tilde{g}}
\newcommand{\GLm}{\mathrm{GL}_m}
\newcommand{\convt}{\tilde{*}}
\newcommand{\mset}{\{1,\ldots, m\}}
\newcommand{\Lset}{\{1,\ldots,L\}}
\newcommand{\Lmoset}{\{1,\ldots,L-1\}}
\newcommand{\diag}[1]{\mathrm{diag}(#1)}
\newcommand{\Ph}{\widehat{P}}
\newcommand{\Loloc}{L^1_{\mathrm{loc}}}
\newcommand{\Texp}[1]{\overleftarrow{\exp}\left( #1 \right)}
\newcommand{\ip}[2]{\left\langle #1, #2\right\rangle}
\newcommand{\fh}{\hat{f}}
\newcommand{\PDO}{\mathrm{\Psi DO}}
\newcommand{\lap}{\bigtriangleup}

\newcommand{\CjN}[3][]{
\ifthenelse{\isempty{#1}}
{\left\| #3 \right\|_{C^{#2}}}
{\left\| #3 \right\|_{C^{#2}(#1)}}
}

\newcommand{\dw}[1][]{
\ifthenelse{\isempty{#1}}
{\frac{\partial}{\partial w}}
{\frac{\partial^{#1}}{\partial w^{#1}}}
}

\newcommand{\dt}[1][]{
\ifthenelse{\isempty{#1}}
{\frac{\partial}{\partial t}}
{\frac{\partial^{#1}}{\partial t^{#1}}}
}

\newcommand{\dz}[1][]{
\ifthenelse{\isempty{#1}}
{\frac{\partial}{\partial z}}
{\frac{\partial^{#1}}{\partial z^{#1}}}
}

\newcommand{\du}[1][]{
\ifthenelse{\isempty{#1}}
{\frac{\partial}{\partial u}}
{\frac{\partial^{#1}}{\partial u^{#1}}}
}

\newtheorem{thm}{Theorem}[section]
\newtheorem{cor}[thm]{Corollary}
\newtheorem{prop}[thm]{Proposition}
\newtheorem{lemma}[thm]{Lemma}
\newtheorem{conj}[thm]{Conjecture}
\newtheorem{prob}[thm]{Problem}

\theoremstyle{remark}
\newtheorem{rmk}[thm]{Remark}

\theoremstyle{definition}
\newtheorem{defn}[thm]{Definition}

\theoremstyle{definition}
\newtheorem{assumption}[thm]{Assumption}

\theoremstyle{remark}
\newtheorem{example}[thm]{Example}

\crefformat{equation}{(#2#1#3)}
\crefrangeformat{equation}{(#3#1#4) to~(#5#2#6)}
\crefmultiformat{equation}{(#2#1#3)}%
{ and~(#2#1#3)}{, (#2#1#3)}{ and~(#2#1#3)}

\Crefformat{equation}{(#2#1#3)}
\Crefrangeformat{equation}{(#3#1#4) to~(#5#2#6)}
\Crefmultiformat{equation}{(#2#1#3)}%
{ and~(#2#1#3)}{, (#2#1#3)}{ and~(#2#1#3)}

\numberwithin{equation}{section}

\title{Differential Equations with a Difference Quotient}
\author{Brian Street\footnote{The author was partially supported by NSF DMS-1401671.}}
\date{}

\maketitle


\begin{abstract}
The purpose of this paper is to study a class of ill-posed differential equations.
In some settings, these differential equations exhibit uniqueness but not existence,
while in others they exhibit existence but not uniqueness.
An example of such a differential equation is, for
a polynomial $P$ and continuous functions $f(t,x):[0,1]\times [0,1]\rightarrow \mathbb{R}$,
\begin{equation*}
\frac{\partial}{\partial t} f(t,x) = \frac{ P(f(t,x))-P(f(t,0))}{x}, \quad x>0.
\end{equation*}
These differential equations are related to inverse problems.

\medskip

\end{abstract}

\section{Introduction}
The purpose of this paper is to study a family of ill-posed differential equations.  In some instances, these equations exhibit existence, but not uniqueness.
In other instances, they exhibit uniqueness, but not existence.
The questions studied here can be seen as a family of forward and inverse problems, which in special cases become well-known examples
from the literature.  This is discussed more below and detailed in \cref{SectionForInv}.

In this introduction, we informally state the main results, and present their relationship to inverse problems.  However, before we enter into the results in full generality, to help the reader understand our somewhat technical results, we give some very simple special cases, where some of the basic ideas already appear in a simple form:

\begin{example}[Existence without uniqueness]\label{ExampleIntroExistence}
Fix $\epsilon_1,\epsilon_0>0$.  We consider the differential equation, defined for functions $f(t,x)\in C([0,\epsilon_1]\times[0,\epsilon_0])$ by
\begin{equation}\label{EqnIntroLinearExist}
\dt f(t,x) =
\frac{f(t,0)-f(t,x)}{x}, \quad x>0.
\end{equation}
We claim that \cref{EqnIntroLinearExist} has existence:  i.e., given $f_0(x)\in C([0,\epsilon_0])$, there exists a solution $f(t,x)$ to \cref{EqnIntroLinearExist}
with $f(0,x)=f_0(x)$.  Indeed, given $a(t)\in C([0,\epsilon_1])$ with $a(0)=f_0(0)$ set
\begin{equation}\label{EqnIntroExistSolution}
f(t,x) =
\begin{cases}
 e^{-t/x} f_0(x) + \frac{1}{x}\int_0^t e^{(s-t)/x} a(s)\: ds, & x>0,\\
 a(t), &x=0.
 \end{cases}
\end{equation}
It is immediate to verify that $f(t,x)\in C([0,\epsilon_1]\times [0,\epsilon_0])$ and satisfies \cref{EqnIntroLinearExist}.
Furthermore, this is the unique solution, $f(t,x)$, to \cref{EqnIntroLinearExist} with $f(0,x)=f_0(x)$ \textit{and} $f(t,0)=a(t)$.\footnote{Uniqueness is immediate here, since for $x>0$, if $f(t,0)$ is assumed to be $a(t)$, then \cref{EqnIntroLinearExist} is a standard ODE and standard uniqueness theorems apply.}
Thus, to uniquely determine the solution to \cref{EqnIntroLinearExist} one needs to give both $f(0,x)$ \textit{and} $f(t,0)$.
We call this existence without uniqueness, since there are many solutions corresponding to any initial condition $f(0,x)$--one for each
choice of $a(t)$.
\end{example}

\begin{example}[Uniqueness without existence]\label{ExampleIntroUniqueness}
Fix $\epsilon_1,\epsilon_0>0$.  We consider the differential equation, defined for functions $f(t,x)\in C([0,\epsilon_1]\times[0,\epsilon_0])$ by
\begin{equation}\label{EqnIntroLinearUnique}
\dt f(t,x) =
\frac{f(t,x)-f(t,0)}{x}, \quad x>0.
\end{equation}
We claim that \cref{EqnIntroLinearUnique} has uniqueness: i.e., if $f(t,x), g(t,x)\in C([0,\epsilon_1]\times [0,\epsilon_0])$ both
satisfy \cref{EqnIntroLinearUnique} and $f(0,x)=g(0,x)$, $\forall x$, then $f(t,x)=g(t,x)$, $\forall t,x$.
Indeed, suppose $f(t,x)$ satisfies \cref{EqnIntroLinearUnique}.  Then, by reversing time, treating $f(\epsilon_1,x)$ as our initial condition,
and treating $a(t):=f(t,0)$ as a given function, we may solve the differential equation \cref{EqnIntroLinearUnique}, for $x>0$, to see
\begin{equation}\label{EqnIntroUniqueInitial}
f(0,x)=e^{-\epsilon_1/x}f(\epsilon_1,x) + \frac{1}{x}\int_0^{\epsilon_1} e^{-u/x} a(u)\: du,\quad x>0.
\end{equation}
From \cref{EqnIntroUniqueInitial} uniqueness follows.  Indeed, if $f(t,x)$ and $g(t,x)$ are two solutions to \cref{EqnIntroLinearUnique}
with $f(0,x)=g(0,x)$ $\forall x$, then \cref{EqnIntroUniqueInitial} shows
\begin{equation*}
\frac{1}{x}\int_0^{\epsilon_1} e^{-u/x} f(u,0)\: du = \frac{1}{x}\int_0^{\epsilon_1} e^{-u/x} g(u,0)\: du + O(e^{-\epsilon_1/x}).
\end{equation*}
It then follows (see \cref{CorLaplaceAppendixUsed}) that $f(t,0)=g(t,0)$ $\forall t$.  With $f(t,0)=g(t,0)$ in hand, \cref{EqnIntroLinearUnique} is a standard ODE
for $x>0$ and it follows that $f(t,x)=g(t,x)$ $\forall t,x$.  This proves uniqueness.
Furthermore, \cref{EqnIntroUniqueInitial} shows that \cref{EqnIntroLinearUnique} does not have existence:  not every initial condition gives rise to a solution.
In fact, every initial condition that does give rise to a solution must be of the form given by \cref{EqnIntroUniqueInitial}, for some continuous functions
$a(t)$ and $f(\epsilon_1,x)$.  I.e., the initial condition must be of Laplace transform type, modulo an appropriate error.
Furthermore, it is easy to see that for such an initial condition, there exists a solution.
Hence, we have exactly characterized the initial conditions
which give rise to a solution to \cref{EqnIntroLinearUnique}.
\end{example}


The goal of this paper is to extend the above ideas to a nonlinear setting.  Consider the following simplified example.
\begin{example}\label{ExampleIntroOneD}
Let $P(y) = \sum_{j=1}^D c_j y^{j}$ be a polynomial without constant term.  Consider the differential equation,
defined for functions $f(t,x)\in C([0,\epsilon_1]\times[0,\epsilon_0])$, given by
\begin{equation}\label{EqnIntroNonlinearConstant}
\dt f(t,x) = \frac{P(f(t,x))-P(f(t,0))}{x}, \quad x>0.
\end{equation}
\begin{itemize}
\item (Uniqueness without existence) If we restrict our attention to solutions $f(t,x)$ with $P'(f(t,0))>0$ $\forall t$ and we insist that $f(t,0)\in C^2([0,\epsilon_1])$, then \cref{EqnIntroNonlinearConstant}
has uniqueness (but not existence).  I.e., if $f(t,x),g(t,x)\in C([0,\epsilon_1]\times[0,\epsilon_0])$ are two solutions to \cref{EqnIntroNonlinearConstant}
with $f(0,x)=g(0,x)$ $\forall x$, $P'(f(t,0)), P'(g(t,0))>0$ $\forall t$, and $f(t,0),g(t,0)\in C^2([0,\epsilon_1])$, then $f(t,x)=g(t,x)$ $\forall t,x$.
However, not every initial condition gives rise to a solution.  See \cref{SectionResUnique}.
This generalizes\footnote{Since we insisted $f(t,0)\in C^2$, this is not strictly a generalization of \cref{ExampleIntroUniqueness},
however it does generalize the basic ideas of \cref{ExampleIntroUniqueness}.  A similar remark holds for the next part where
we discuss existence without uniqueness.} \cref{ExampleIntroUniqueness} where $P(y)=y$ and therefore $P'(y)\equiv 1>0$.

\item (Existence without uniqueness)  Given $f_0(x)\in C([0,\epsilon_0])$ and $a(t)\in C^2([0,\epsilon_1])$ with $a(0)=f_0(0)$ and $P'(a(t))<0$ $\forall t$,
there exists $\delta>0$ and a unique solution $f(t,x)\in C([0,\epsilon_1]\times [0,\delta])$ to \cref{EqnIntroNonlinearConstant} satisfying
$f(0,x)=f_0(x)$ \textit{and} $f(t,0)=a(t)$.  See \cref{SectionResExist}.  This generalizes \cref{ExampleIntroExistence} where $P(y)=-y$ and therefore $P'(y)\equiv -1<0$.
\end{itemize}
In short, if one has $P'(f(t,0))>0$ $\forall t$, one has uniqueness but not existence, and if one has $P'(f(t,0))<0$ $\forall t$, one has
existence but not uniqueness.
\end{example}

We now turn to the more general setting of our main results.  Fix $m\in \N$ and $\epsilon_0,\epsilon_1>0$.
For $t\in [0,\epsilon_1]$, $x\in [0,\epsilon_0]$, and $y,z\in \R^m$, let $P(t,x,y,z)$ be a polynomial in $y$ given by
\begin{equation*}
P(t,x,y,z)=\sum_{j=1}^m \sum_{|\alpha|\leq D} c_{\alpha,j}(t,x,z) y^{\alpha} e_j,
\end{equation*}
where $e_j\in \R^m$ denotes the $j$th standard basis element.
For $f(t,x)\in C([0,\epsilon_1]\times [0,\epsilon_0];\R^m)$ we consider the differential equation
\begin{equation}\label{EqnIntroFullEqn}
\dt f(t,x) = \frac{P(t,x,f(t,x),f(t,0))-P(t,0,f(t,0),f(t,0))}{x}, \quad x>0.
\end{equation}
We state our assumptions more rigorously in \cref{SectionResults}, but we assume:
\begin{itemize}
\item $c_{\alpha,j}(t,x,z) = \frac{1}{x}\int_0^\infty e^{-w/x} b_{\alpha,j}(t,w,z)\: dw$, where the $b_{\alpha,j}(t,w,z)$ have a certain prescribed level of smoothness.
\item We consider only solutions $f(t,x)\in C([0,\epsilon_1]\times [0,\epsilon_0];\R^m)$ such that $f(t,0)\in C^2([0,\epsilon_1];\R^m)$.
\item For $y\in \R^m$, set $\sM_y(t):=d_y P(t,0,y,y)$, so that $\sM_y(t)$ is an $m\times m$ matrix.
We consider only solutions $f(t,x)$ such that
there exists an invertible matrix $R(t)$ which is $C^1$ in $t$ and such that $R(t)\sM_{f(t,0)}(t)R(t)^{-1}$ is a diagonal matrix.  When $m=1$, this is automatic.
\end{itemize}

Under the above assumptions, we prove the following:
\begin{itemize}
\item (Uniqueness without existence)  Under the above hypotheses, if $\sM_{f(t,0)}(t)$ is assumed to have all strictly positive eigenvalues, then \cref{EqnIntroFullEqn}
has uniqueness, but not existence.  I.e., if $f(t,x),g(t,x)\in C([0,\epsilon_1]\times [0,\epsilon_0];\R^m)$ are solutions to \cref{EqnIntroFullEqn} which satisfy all of the above hypothesis and such that
the eigenvalues of $\sM_{f(t,0)}(t)$ and $\sM_{g(t,0)}(t)$ are strictly positive, for all $t$, then if $f(0,x)=g(0,x)$ $\forall x$, we have $f(t,x)=g(t,x)$ $\forall t,x$.
Furthermore, in this situation we prove stability estimates.  Finally, in analogy to \cref{ExampleIntroUniqueness}, we will see that only certain initial conditions
give rise to solutions.  See \cref{SectionResUnique}.

\item (Existence without uniqueness)  Suppose $f_0(x)\in C([0,\epsilon_0];\R^m)$ and $A(t)\in C^2([0,\epsilon_1];\R^m)$ are given
such that $f_0(0)=A(0)$ and $\sM_{A(t)}(t)$ has all strictly negative eigenvalues.  Suppose further that there exists an invertible matrix $R(t)$, which is $C^1$ in $t$ such that
$R(t)\sM_{A(t)} R(t)^{-1}$ is a diagonal matrix.  Then we show that there exists $\delta>0$ and a unique function $f(t,x)\in C([0,\epsilon_1]\times[0,\delta];\R^m)$ such
that $f(0,x)=f_0(x)$, $f(t,0)=A(t)$, and $f(t,x)$ solves \cref{EqnIntroFullEqn}.  See \cref{SectionResExist}.
\end{itemize}

The main idea is the following.  If $f(t,x)$ were assumed to be of Laplace transform type, $f(t,x)=\frac{1}{x}\int_0^\infty e^{-w/x} A(t,w)\: dw$,
then \cref{EqnIntroFullEqn} can be restated as a partial differential equation on $A(t,w)$--and this partial differential
equation is much easier to study.  As exemplified in \cref{ExampleIntroExistence,ExampleIntroUniqueness}, not every solution is of Laplace
transform type.  However, we will show (under the above discussed hypotheses) that every solution is of Laplace transform type modulo an error which can be controlled.  Once this is done, the above results follow. 

\subsection{Motivation and relation to inverse problems}
It is likely that the methods of this paper are the most interesting aspect.
The differential equations in this paper seem to not fall under any current methods (the equations
are too unstable),
and the methods in this paper are largely new.
Moreover, as we will see, special cases of the above appear in some inverse problems.
Furthermore, there are other (harder) inverse problems where differential equations
similar to (but more complicated than) the ones studied in this paper appear.  For example, we will see in
\cref{AppendixDefineCalderon} that the anisotropic version
of the famous Calder\'on problem involves a ``non-commutative'' version
of some of these differential equations.  We hope that the ideas
in this paper might shed light on such questions--and, indeed, one of our motivation
for these
results is as a simpler model case for full anisotropic version of the Calder\'on problem.

We briefly outline the relationship between these results and inverse problems;
these ideas are discussed in greater detail in \cref{SectionForInv}.
We begin by explaining that the results in this paper can be thought of as a
class of forward and inverse problems.
For simplicity, consider the setting in \cref{ExampleIntroOneD}, with $\epsilon_0=\epsilon_1=1$.  Thus,
we are given a polynomial without constant term, $P(y) = \sum_{j=1}^D c_j y^j$.  We consider
the differential equation, for functions $f(t,x)$, given by
\begin{equation}\label{EqnMotivationOneDEqn}
\dt f(t,x) = \frac{P(f(t,x))-P(f(t,0))}{x}, \quad x>0.
\end{equation}

\noindent \textbf{Forward Problem:}  Given a function $f_0(x)\in C([0,1])$ and $a(t)\in C^2([0,1])$
with $P'(a(t))<0$, $\forall t$ and $f_0(0)=a(0)$, the results below imply that there exists $\delta>0$
and a unique solution $f(t,x)\in C([0,1]\times[0,\delta])$ to \cref{EqnMotivationOneDEqn} with
$f(0,x)=f_0(x)$ and $f(t,0)=a(t)$.  
\begin{center}
The forward problem is the map $(f_0(\cdot),a(\cdot))\mapsto f(1,\cdot)$.
\end{center}

\noindent\textbf{Inverse Problem:}  The inverse problem is given $f(1,\cdot)$, as above, to find $f_0(\cdot)$ and $a(\cdot)$. 

To see how the inverse problem relates to the main results of the paper, let $f(t,x)$ be the solution
as above.  Set $g(t,x) = f(1-t,x)$.  If $Q(y)=-P(y)$, then $g(t,0)=a(t)$ and  $g(t,x)$ satisfies 
\begin{equation}\label{EqnMotivationForg}
\dt g(t,x) = \frac{Q(g(t,x))-Q(g(t,0))}{x}, \quad x>0.
\end{equation}
Also, $Q'(g(t,0))>0$, $\forall t$.  The main results of this paper imply \cref{EqnMotivationForg} has
uniqueness in this setting:  $g(0,x)\in C([0,\delta])$ uniquely determines $g(t,x)\in C([0,1]\times [0,\delta])$.
Since $g(t,x)=f(1-t,x)$, $f(1,x)\in C([0,\delta])$ uniquely determines both $f_0\big|_{[0,\delta]}$ and $a(t)$.
Thus, the inverse problem has uniqueness.
In short, the map $(f_0\big|_{[0,\delta]}(\cdot), a(\cdot))\mapsto f(1,\cdot)$ is injective (though it
is far from surjective as we explain below).

We go further than just proving existence and uniqueness, though.  We have:
\begin{itemize}
\item In the forward problem, we do the following (see \cref{SectionResExist}):
\begin{itemize}
\item Beyond just proving existence, we show that every solution $f(t,x)$
must be of Laplace transform type, modulo an appropriate error, for every $t>0$.  This is despite
the fact that the initial condition, $f(0,x)=f_0(x)$, can be any continuous function
with $P'(f_0(0))<0$.

\item We reduce the problem to a more stable PDE, so that solutions can be more easily studied.
\end{itemize}

\item In the inverse problem, we do the following (see \cref{SectionResUniqueMore}):
\begin{itemize}
\item We characterize the initial conditions $g(0,x)$ which give rise to solutions to \cref{EqnMotivationForg}.
In other words, we characterize the image of the map $(f_0(\cdot),a(\cdot))\mapsto f(1,\cdot)$.
We see that all such functions are of Laplace transform type, modulo an appropriate error.

\item We give a procedure to reconstruct $a(t)$ and $f_0\big|_{[0,\delta]}$ from $f(1,\cdot)$.  This is
necessarily unstable, but we reduce the instability to the instability of the Laplace transform, which is
well understood.

\item We prove a kind of stability for the inverse problem.  Namely if one has
two solutions $g_1(t,x)$ and $g_2(t,x)$ to \cref{EqnMotivationForg} such
that $g_1(0,x)-g_2(0,x)$ vanishes sufficiently quickly as $x\downarrow 0$,
then $g_1(t,0)=g_2(t,0)$ on a neighborhood of $0$ (the size of the neighborhood
depends on how quickly $g_1(0,x)-g_2(0,x)$ vanishes in a way which is made precise).
In other words if one only knows $f(1,x)$ modulo functions which vanish sufficiently
quickly at $0$, one can still reconstruct $a(t)$ on a neighborhood of $t=1$,
in a way which we make quantitative.
\end{itemize}
\end{itemize}

Some special cases of the main results in this paper can be interpreted as some standard
inverse problems in the following way:
\begin{itemize}
\item When $P(y)=-y$, we saw in \cref{ExampleIntroExistence,ExampleIntroUniqueness}
that the forward problem is essentially taking the Laplace transform, and the inverse problem
is essentially taking the inverse Laplace transform.  See \cref{SectionForInvLaplace} for more details on this.
As a consequence, the results in this paper can be interpreted as nonlinear analogs of the Laplace transform.

\item In our main results, we allow the coefficients of the polynomial to be functions of $x$.  We will
see in \cref{SectionForInvInvSpec} that the special case of $P(x,y) = -y-x^2y^2$
is closely related to Simon's approach \cite{SimonANewApproachToInverseSpectalTheoryI}
to the theorem of Borg \cite{BorgUniquensesTheoremsInTheSpectralTheory} and Mar\v{c}enko \cite{MarchenkoSomeQuestions} that
the principal $m$-function for a finite interval or half-line Schr\"odinger operator determines
the potential.

\item In our main results, we allow $f$ to be vector valued, and also allow the coefficients to depend on
$f(t,0)$.  By doing this, we see in \cref{AppendixTransInvCalderon} that the translation invariant version
of the anisotropic version of Calder\'on's inverse problem can be seen in this framework.
\end{itemize}

Thus, the results in this paper can be viewed as a family of inverse problems which generalize
and unify the above examples, and for which we have good results on uniqueness,
characterization of solutions, a reconstruction procedure, and stability estimates.

Furthermore, as argued in \cref{AppendixDefineCalderon}, a non-commutative analog\footnote{Achieved
by replacing functions with pseudodifferential operators:  here the frequency plays the role
that $x^{-1}$ plays in our main results.}
of these equations arise in the full anisotropic version of the Calder\'on problem.
Thus, a special case of results in this paper can be seen as a simplified model case for the
full Calder\'on problem.  Moreover, by replacing functions in our results with
pseudodifferential operators, one gives rise to an entire family of conjectures
which generalize the Calder\'on problem.

\subsection{Selected Notation}
\begin{itemize}
\item All functions take their values in real vector spaces
or spaces of real matrices.  Other than in \cref{AppendixLaplace}, there are no complex numbers in this paper.
\item Let $\epsilon_1,\epsilon_2>0$.  For $n_1,n_2\in \N$, we write $b(t,w)\in C^{n_1,n_2}([0,\epsilon_1]\times [0,\epsilon_2])$
if for $0\leq j\leq n_1$, $0\leq k\leq n_2$, $\dt[j]\dw[k]b(t,w)\in C([0,\epsilon_1]\times [0,\epsilon_2])$.
If $U\subseteq \R^m$ is open, and $n_3\in \N$, we write $c(t,w,z)\in C^{n_1,n_2,n_3}([0,\epsilon_1]\times [0,\epsilon_2]\times U)$ if
for $0\leq j\leq n_1$, $0\leq k\leq n_2$, and $0\leq |\alpha|\leq n_3$, we have $\dt[j]\dw[k]\dz[\alpha] c(t,w,z)\in C([0,\epsilon_1]\times [0,\epsilon_2]\times U)$.
We define the norms
\begin{equation*}
\CjN{n_1,n_2}{b}:=\sum_{j=0}^{n_1} \sum_{k=0}^{n_2}\sup_{t,w} \left|\dt[j]\dw[k] b(t,w)\right|,
\end{equation*}
\begin{equation*}
 \CjN{n_1,n_2,n_3}{c}:= \sum_{j=0}^{n_1} \sum_{k=0}^{n_2}\sum_{ |\alpha|\leq n_3} \sup_{t,w,z}\left| \dt[j]\dw[k]\dz[\alpha] c(t,w,z) \right|.
\end{equation*}

\item If $V\subseteq \R^n$ is open, and $U\subseteq\R^m$, we write $C^j(V;U)$ to be the usual space of $C^j$ functions on $V$
taking values in $U$.  We use the norm
\begin{equation*}
\CjN[V;U]{j}{g}:= \sum_{|\alpha|\leq j} \sup_{z\in V} \left|\dz[\alpha] g(z)\right|.
\end{equation*}

\item We write $\M^{m\times n}$ to be the space of $m\times n$ real matrices.  We write $\GLm$ to be the space of $m\times m$
real, invertible matrices.

\item For $a(w),b(w)\in C([0,\epsilon_2])$ we write
\begin{equation}\label{EqnNotationConvt}
(a\convt b)(w):=\int_0^w a(w-r) b(r)\: dr\in C([0,\epsilon_2]).
\end{equation}
Note that $\convt$ is commutative and associative.

\item If $A(w)\in C([0,\epsilon_2];\R^m)$ and $\alpha=(\alpha_1,\ldots, \alpha_m)\in \N^m$ is a multi-index, we write
\begin{equation*}
\convt^{\alpha} A = \underbrace{A_1\convt \cdots \convt A_1}_{\alpha_1\text{ terms}}\convt
\cdots\convt \underbrace{A_j\convt\cdots\convt A_j}_{\alpha_j\text{ terms}}\convt
 \cdots \convt \underbrace{A_m\convt\cdots \convt A_m}_{\alpha_m\text{ terms}}.
\end{equation*}
and with a slight abuse of notation, if $|\alpha|=0$ and $b(w)$ is another function, we write $b \convt (\convt^{\alpha} A) = b$.

\item If $A(t,w)$ is a function of $t$ and $w$, we write $\dot{A}=\dt A$ and $A'=\dw A$.

\item For $\lambda_1,\ldots,\lambda_m\in \R$, we write $\diag{\lambda_1,\ldots, \lambda_m}$ to denote the $m\times m$ diagonal
matrix with diagonal entries $\lambda_1,\ldots, \lambda_m$.

\item We write $A\lesssim B$ to mean $A\leq CB$, where $C$ depends only on certain parameters.  It will always be clear from context
what $C$ depends on.

\item We write $a\wedge b$ to mean $\min\{a,b\}$.
\end{itemize} 

\section{Statement of Results}\label{SectionResults}
Fix $m\in \N$, $\epsilon_0,\epsilon_1,\epsilon_2\in (0,\infty)$, $U\subseteq \R^m$ open, and $D\in \N$.
For $j\in \mset$, $\alpha\in \N^m$ a multi-index with $|\alpha|\leq D$, let
\begin{equation*}
b_{\alpha,j}(t,w,z)\in C^{0,3,0}([0,\epsilon_1]\times [0,\epsilon_2]\times U),
\end{equation*}
with $b_{\alpha,j}(t,0,z), \left(\dw b_{\alpha,j}\right)(t,0,z)\in C^1([0,\epsilon_1]\times U)$.
Define $c_{\alpha,j}(t,x,z)\in C([0,\epsilon_1]\times [0,\epsilon_0]\times U)$ by
\begin{equation*}
c_{\alpha,j}(t,x,z):=\frac{1}{x}\int_0^{\epsilon_2} e^{-w/x} b_{\alpha,j}(t,w,z)\: dw.
\end{equation*}
We assume there is a $C_0<\infty$ with
\begin{equation*}
\CjN[{[0,\epsilon_1]\times[0,\epsilon_2]\times U}]{0,3,0}{b_{\alpha,j}}, \CjN[{[0,\epsilon_1]\times U}]{1}{b_{\alpha,j}(t,0,z)}, \CjN[{[0,\epsilon_1]\times U}]{1}{\frac{\partial}{\partial w}b_{\alpha,j}(t,0,z)}\leq C_0.
\end{equation*}

\begin{example}\label{ExampleResultsPolynomial}
Because $\frac{1}{x}\int_0^{\epsilon_2} e^{-w/x} \frac{w^l}{l!}\: dw=x^l+e^{-\epsilon_2/x} G(x)$, with $G\in C([0,\infty))$,
any polynomial in $x$ can be written in the form covered by the $c_{\alpha,j}$, modulo error terms of the form $e^{-\epsilon_2/x} G(x)$, $G\in C([0,\infty))$.
The results below are invariant under such error terms, so polynomials in $x$ can be considered as a special case of the $c_{\alpha,j}$.
\end{example}

Define $P(t,x,y,z):=(P_1(t,x,y,z),\ldots, P_m(t,x,y,z))$, where for $y\in \R^m$,
\begin{equation*}
P_j(t,x,y,z)=\sum_{|\alpha|\leq D} c_{\alpha,j}(t,x,z) y^{\alpha}.
\end{equation*}
Let $V\subseteq \R^m$ be an open set with $U\subseteq V$.  Let $G(t,x,y,z)\in C([0,\epsilon_1]\times [0,\epsilon_0]\times V\times U;\R^m)$
be such that for every $\gamma\in (0,\epsilon_2)$, $G(t,x,y,z)=e^{-\gamma/x} G_\gamma(t,x,y,z)$,
where $G_{\gamma}(t,x,y,z)\in C([0,\epsilon_1]\times[0,\epsilon_0]\times V\times U;\R^m)$ satisfies
for any compact sets $K_1\Subset U$, $K_2\Subset V$,
\begin{equation*}
\sup_{\substack{t\in [0,\epsilon_1],x\in[0,\epsilon_0],z\in K_1\\ y_1,y_2\in K_2, y_1\ne y_2}} \frac{\left|G_\gamma(t,x,y_1,z)-G_\gamma(t,x,y_2,z)\right|}{|y_1-y_2|}<\infty.
\end{equation*}

We will be considering the differential equation, defined for $f(t,x)\in C([0,\epsilon_1]\times[0,\epsilon_0];V)$
with $f(t,0)\in C([0,\epsilon_1];U)$,
\begin{equation}\label{EqnResultMainEqn}
\dt f(t,x) = \frac{P(t,x,f(t,x), f(t,0))-P(t,0,f(t,0),f(t,0))}{x}+G(t,x,f(t,x),f(t,0)), \quad x>0.
\end{equation}

Corresponding to $P(t,x,y,z)$, for $\delta\in (0,\epsilon_2]$ and $A\in C^1([0,\delta];\R^m)$, we define
\begin{equation*}
\Ph(t, A(\cdot),z)(w) = \left( \Ph_1(t,A(\cdot),z)(w),\ldots, \Ph_m(t,A(\cdot),z)(w)\right)
\end{equation*}
by
\begin{equation*}
\Ph_j(t,A(\cdot),z)(w)=\sum_{|\alpha|\leq D} \dw[|\alpha|+1] \left(  b_{\alpha,j}(t,\cdot,z) \convt (\convt^{\alpha} A)   \right)(w).
\end{equation*} 

\subsection{Existence without Uniqueness}\label{SectionResExist}


\begin{thm}\label{ThmResultExist}
Suppose $f_0(x)\in C([0,\epsilon_0];V)$ and $A_0(t)\in C^2([0,\epsilon_1];U)$ are given, with $f_0(0)=A_0(0)$.
Set $\sM(t):=-d_y P(t,0,A_0(t),A_0(t))$.\footnote{Notice the minus sign in the definition of $\sM(t)$.  This is in contrast to the notation in the introduction,
which lacked the minus sign.}
We suppose that there exists $R(t)\in C^1([0,\epsilon_1];\GLm)$
with
\begin{equation*}
R(t)\sM(t)R(t)^{-1} = \diag{\lambda_1(t),\ldots, \lambda_m(t)},
\end{equation*}
where $\lambda_j(t)>0$ for all $j,t$.
Then, there exists $\delta_0>0$ and a unique solution $f(t,x)\in C([0,\epsilon_1]\times [0,\delta_0];\R^m)$ to \cref{EqnResultMainEqn},
satisfying $f(0,x)=f_0(x)$ and $f(t,0)=A_0(t)$.
\end{thm}

\begin{rmk}
As in the introduction, we call this existence without uniqueness because one has to specify both $f(0,x)$ and $f(t,0)$ (as opposed to just $f(0,x)$).
\end{rmk}

Beyond proving existence, we can show that the solution given in \cref{ThmResultExist} is of Laplace transform type, modulo an appropriate error,
as shown in the next theorem.

\begin{thm}\label{ThmResultExistCharaterize}
Take the same assumptions as in \cref{ThmResultExist}, and let $f(t,x)$ be the unique solution guaranteed by \cref{ThmResultExist}.
Take $c_0, C_1,C_2,C_3, C_4>0$ such that $\min_{t,j} \lambda_j(t)\geq c_0>0$, $\CjN{1}{R}\leq C_1$, $\CjN{1}{R^{-1}}\leq C_2$, $\CjN{1}{\sM^{-1}}\leq C_3$, $\CjN{2}{A_0}\leq C_4$.  Then, there exists $\delta=\delta(m,D,c_0,C_0,C_1,C_2,C_3,C_4)>0$ and $A(t,w)\in C^{0,2}([0,\epsilon_1]\times [0,\delta\wedge \epsilon_2];\R^m)$ such that
\begin{equation*}
\dt A(t,w) = \Ph(t,A(t,\cdot),A(t,0))(w), \quad A(t,0)=A_0(t),
\end{equation*}
and such that if $\lambda_0(t)=\min_j \lambda_j(t)$, then for all $\gamma\in [0,1)$,
\begin{equation}\label{EqnResulExistChar}
f(t,x)=\frac{1}{x}\int_0^{\delta\wedge \epsilon_2} e^{-w/x} A(t,w)\: dw + O\left( e^{-\gamma(\delta\wedge \epsilon_2)/x} + e^{-(\gamma/x)\int_0^t \lambda_0(s)\: ds} \right),\quad x\in (0,\delta_0],
\end{equation}
where the implicit constant in the $O$ in \cref{EqnResulExistChar} does not depend on $(t,x)\in [0,\epsilon_1]\times (0,\delta_0]$.
Furthermore, the representation \cref{EqnResulExistChar} is unique in the following sense.  Fix $t_0\in [0,\epsilon_1]$.
Suppose there exists $0<\delta'<\delta\wedge \epsilon_2\wedge \left(\int_0^{t_0} \lambda_0(s)\: ds\right)$ and
$B\in C([0,\delta'];\R^m)$ with
\begin{equation*}
f(t_0,x) =\frac{1}{x}\int_0^{\delta'} e^{-w/x} B(w)\: dw + O\left(e^{-\delta'/x}\right),\text{ as }x\downarrow 0.
\end{equation*}
Then, $A(t_0,w)=B(w)$, $\forall w\in [0,\delta']$.
\end{thm} 

\subsection{Uniqueness without Existence}\label{SectionResUnique}
In addition to the above assumptions, for the next result we assume for every compact set $K\Subset U$,
\begin{equation}\label{EqnResUniqueAddb}
\begin{split}
\sup_{\substack{t\in [0,\epsilon_1], w\in [0,\epsilon_2]  \\ z_1,z_2\in K, z_1\ne z_2} } \frac{|b_{\alpha,j}(t,w,z_1)-b_{\alpha,j}(t,w,z_2)|} {|z_1-z_2|}&<\infty, \\
\sup_{\substack{t\in [0,\epsilon_1], w\in [0,\epsilon_2]  \\ z_1,z_2\in K, z_1\ne z_2} } \frac{|\dw b_{\alpha,j}(t,w,z_1)-\dw b_{\alpha,j}(t,w,z_2)|} {|z_1-z_2|}&<\infty.
\end{split}
\end{equation}

\begin{thm}\label{ThmResUnique}
Suppose $f_1(t,x),f_2(t,x)\in C([0,\epsilon_1]\times[0,\epsilon_0];V)$ satisfy $f_j(t,0)\in C^2([0,\epsilon_1];U)$,
both satisfy \cref{EqnResultMainEqn}, and $f_1(0,x)=f_2(0,x)$, $\forall x\in [0,\epsilon_0]$.
Set $\sM_k(t):=d_yP(t,0,f_k(t,0),f_k(t,0))$.  We suppose that there exists $R_k(t)\in C^1([0,\epsilon_1];\GLm)$ with
\begin{equation*}
R_k(t) \sM_k(t) R_k(t)^{-1} = \diag{\lambda_1^k(t),\ldots, \lambda_m^k(t)},
\end{equation*}
where $\lambda_j^k(t)>0$, $\forall j\in \mset$, $t\in [0,\epsilon_1]$.  Then, $f_1(t,x)=f_2(t,x)$, $\forall t\in [0,\epsilon_1], x\in [0,\epsilon_0]$.
\end{thm}

\cref{ThmResUnique} shows uniqueness, but we will show more.   We will further investigate the following questions:
\begin{itemize}
\item Stability:  If $f_1(0,x)-f_2(0,x)$ vanishes sufficiently quickly at $0$, and under the hypotheses of \cref{ThmResUnique}, we will prove that $f_1(t,0)$ and $f_2(t,0)$ agree for small $t$, and we will make this quantitative.  See \cref{ThmResUniqueStability}.
\item Reconstruction: Given the initial condition $f(0,x)$ for \cref{EqnResultMainEqn}, and under the hypotheses of \cref{ThmResUnique}, we will
show how to reconstruct the solution $f(t,x)$, for all $t$.  This is an unstable process, but we will reduce the instability to that of inverting the Laplace transform, which is well understood.  See \cref{RmkResUniqueReconstruct}.
\item Characterization:  We will show that if $f(t,x)$ is a solution to \cref{EqnResultMainEqn}, and under the hypotheses of \cref{ThmResUnique}, then
$f(t,x)$ must be of Laplace transform type, modulo an appropriate error term.  In particular, only initial conditions $f(0,x)$ which are of Laplace transform
type modulo an appropriate error give rise to solutions.  See \cref{ThmResCharacterize,RmkUniquefDeterminesA}.
\end{itemize}

We now turn to making these ideas more precise. 

\subsubsection{Stability, Reconstruction, and Characterization}\label{SectionResUniqueMore}
For our first result, we take $P$ in the start of this section, but we drop the assumption \cref{EqnResUniqueAddb}.

\begin{thm}[Charaterization]\label{ThmResCharacterize}
Suppose $f(t,x)\in C([0,\epsilon_1]\times [0,\epsilon_0];\R^m)$ is such that $\forall \gamma\in [0,\epsilon_2)$,
\begin{equation*}
\dt f(t,x) = \frac{P(t,x,f(t,x),f(t,0))-P(t,0,f(t,0),f(t,0))}{x} +O(e^{-\gamma/x}), \quad x\in [0,\epsilon_0),
\end{equation*}
where the implicit constant in $O$ is independent of $t,x$.  We suppose
\begin{itemize}
\item $f(t,0)\in C^2([0,\epsilon_1];U)$.
\item Set $\sM(t):=d_y P(t,0,f(t,0),f(t,0))$.  We suppose there exists $R(t)\in C^1([0,\epsilon_1];\GLm)$ with
\begin{equation*}
R(t)\sM(t)R(t)^{-1} = \diag{\lambda_1(t),\ldots, \lambda_m(t)},
\end{equation*}
where $\lambda_j(t)>0$, for all $j,t$.
\end{itemize}
Take $c_0, C_1,C_2,C_3,C_4>0$ such that $\min_{t,j} \lambda_j(t)\geq c_0>0$, $\CjN{1}{R}\leq C_1$, $\CjN{1}{R^{-1}}\leq C_2$, $\CjN{1}{\sM^{-1}}\leq C_3$,
$\CjN{2}{f(\cdot,0)}\leq C_4$.  Then, there exists $\delta=\delta(m,D,c_0,C_0,C_1,C_2,C_3,C_4)>0$
and $A(t,w)\in C^{0,2}([0,\epsilon_1]\times [0,\delta\wedge \epsilon_2];\R^m)$ such that
\begin{equation}\label{EqnResUniqueADiffEq}
\dt A(t,w) = \Ph(t,A(t,\cdot),A(t,0)), \quad A(t,0)=f(t,0),
\end{equation}
and such that if $\lambda_0(t)=\min_{j} \lambda_j(t)$, then $\forall \gamma\in (0,1)$,
\begin{equation}\label{EqnResultCharacterIsLaplace}
f(t,x)=\frac{1}{x}\int_0^{\delta\wedge \epsilon_2} e^{-w/x} A(t,w)\: dw + O\left( e^{-\gamma(\delta\wedge \epsilon_2)/x} + e^{-(\gamma/x)\int_0^{\epsilon_1-t} \lambda_0(s)\: ds} \right),
\end{equation}
where the implicit constant in $O$ is independent of $t,x$.
Furthermore, the representation in \cref{EqnResultCharacterIsLaplace} of $f(t,x)$ is unique in the following sense.
Fix $t_0\in [0,\epsilon_0]$.  Suppose there exists $0<\delta'<\delta\wedge \epsilon_2\wedge \int_0^{\epsilon_1-t_0}\lambda(s)\: ds$
and $B\in C([0,\delta'];\R^m)$ with
\begin{equation*}
f(t_0,x) = \frac{1}{x}\int_0^{\delta'} e^{-w/x} B(w)\: dw + O\left(e^{-\delta'/x}\right),\text{ as }x\downarrow 0.
\end{equation*}
Then, $A(t_0,w)=B(w)$, $\forall w\in [0,\delta']$.
\end{thm}

\begin{rmk}\label{RmkUniquefDeterminesA}
By taking $t=0$ in \cref{EqnResultCharacterIsLaplace}, we see that $f(0,x)$ is of Laplace transform type, modulo an error:  $\forall \gamma\in (0,1)$,
\begin{equation*}
f(0,x)=\frac{1}{x}\int_0^{\delta\wedge \epsilon_2} e^{-w/x} A(0,w)\: dw + O\left( e^{-\gamma(\delta\wedge \epsilon_2)/x} + e^{-(\gamma/x)\int_0^{\epsilon_1} \lambda_0(s)\: ds} \right).
\end{equation*}
Thus, under the hypotheses of \cref{ThmResCharacterize}, the only initial conditions that give rise to a solution are of Laplace transform type, modulo an appropriate error.
Furthermore, by taking $t_0=0$ in the last conclusion of \cref{ThmResCharacterize}, we see that $f(0,x)$ uniquely determines
$A(0,w)$.
\end{rmk}

For the remainder of the results in this section, we assume \cref{EqnResUniqueAddb}.

\begin{prop}\label{PropUniqueUniqueA}
The differential equation \cref{EqnResUniqueADiffEq} has uniqueness in the following sense.  Let $\delta'>0$ and $A(t,w),B(t,w)\in C^{0,2}([0,\epsilon_1]\times [0,\delta'];\R^m)$ satisfy
\begin{equation}\label{EqnResUniquePropUniqueEqn}
\dt A(t,w) = \Ph(t, A(t,\cdot),A(t,0))(w), \quad \dt B(t,w) = \Ph(t,B(t,\cdot),B(t,0))(w),
\end{equation}
and $A(0,w)=B(0,w)$ for $w\in [0,\delta']$.  Set $A_0(t)=A(t,0)$, and suppose $A_0(t)\in C^2([0,\epsilon_2];\R^m)$ and set
$\sM(t)= d_y P(t,0,A_0(t),A_0(t))$.  Suppose there exists $R(t)\in C^1([0,\epsilon_1];\GLm)$ with
\begin{equation*}
R(t)\sM(t)R(t)^{-1} =\diag{\lambda_1(t),\ldots, \lambda_m(t)},
\end{equation*}
where $\lambda_j(t)>0$ for all $j,t$.  Set $\gamma_0(t):=\max_j \int_0^t \lambda_j(s)\: ds$, and
\begin{equation*}
\delta_0:=\begin{cases}
\gamma_0^{-1}(\delta'),&\text{if }\gamma_0(\epsilon_1)\geq \delta',\\
\epsilon_1,&\text{else.}
\end{cases}
\end{equation*}
Then, $A(t,0)=B(t,0)$ for $t\in [0,\delta_0]$.
\end{prop}

\begin{rmk}[Reconstruction]\label{RmkResUniqueReconstruct}
\cref{PropUniqueUniqueA} leads us to the reconstruction procedure, which is as follows:
\begin{enumerate}[(i)]
\item Given a solution $f(t,x)$ to \cref{EqnResultMainEqn}, satisfying the assumptions of \cref{ThmResUnique}, we use \cref{ThmResCharacterize} to see
that $f(t,x)$ can be written in the form \cref{EqnResultCharacterIsLaplace}.  In particular, as discussed in \cref{RmkUniquefDeterminesA},
$f(0,x)$ uniquely determines $A(0,w)$.  Extracting $A(0,w)$ from $f(0,x)$ involves taking an inverse Laplace transform, and this step therefore inherits any instability inherent in
the inverse Laplace transform.
\item\label{ItemUniqueADiffRecon} With $A(0,w)$ in hand, and with the knowledge that $A(t,w)$ satisfies \cref{EqnResUniqueADiffEq}, \cref{PropUniqueUniqueA} shows
that $A(0,w)$ uniquely determines $A(t,0)=f(t,0)$ for $0\leq t\leq \delta'$, for some $\delta'$.
\item With $f(t,0)$ in hand, for $x>0$ \cref{EqnResultMainEqn} is a standard ODE, and so uniquely determines $f(t,x)$ for $0\leq t\leq \delta'$.
\item Iterating his procedure gives $f(t,x)$, $\forall t$.
\end{enumerate}
The above procedure reduces the reconstruction of $f(t,x)$ from $f(0,x)$ to the reconstruction of $A(t,w)$ from $A(0,w)$.
As we will see in the proof of \cref{PropUniqueUniqueA}, the differential equation satisfied by $A$ is much more stable than that satisfied by $f$.
In particular, we will be able to prove \cref{PropUniqueUniqueA} by a straightforward application of Gr\"onwall's inequality.
\end{rmk}

\begin{thm}[Stability]\label{ThmResUniqueStability}
Suppose $f_1(t,x), f_2(t,x)\in C([0,\epsilon_1]\times [0,\epsilon_0];\R^m)$ satisfy, for $k=1,2$, $\forall \gamma\in (0,\epsilon_2)$,
\begin{equation*}
\dt f_k(t,x) = \frac{P(t,x,f_k(t,x),f_k(t,0))-P(t,0,f_k(t,0), f_k(t,0))}{x} + O\left(e^{-\gamma/x}\right), \quad x\in (0,\epsilon_0],
\end{equation*}
where the implicit constant in $O$ may depend on $\gamma$, but not on $t$ or $x$.  Suppose, further, for some $r>0$ and
all $s\in [0,r)$,
\begin{equation}\label{EqnUniqueStabilityInitial}
f_1(0,x)=f_2(0,x)+O\left(e^{-s/x}\right).
\end{equation}
We assume the following for $k=1,2$:
\begin{itemize}
\item $f_k(t,0)\in C^2([0,\epsilon_1];U)$.
\item Set $\sM_k(t):= d_y P(t,0,f_k(t,0), f_k(t,0))$.  We suppose that there exists $R_k(t)\in C^1([0,\epsilon_1];\GLm)$ with
\begin{equation*}
R_k(t) \sM_k(t)R_k(t)^{-1}=\diag{\lambda_1^k(t),\ldots, \lambda_m^k(t)},
\end{equation*}
where $\lambda_j^k(t)>0$ $\forall j,t$.
\end{itemize}
Take $c_0,C_1,C_2,C_3,C_4>0$ such that for $k=1,2$, $\min_{t,j}\lambda_j^k(t)\geq c_0>0$, $\CjN{1}{R_k}\leq C_1$, $\CjN{1}{R_k^{-1}}\leq C_2$,
$\CjN{1}{\sM_k^{-1}}\leq C_3$, $\CjN{2}{f_k(\cdot,0)}\leq C_4$.
Set $\gamma_0(t):=\max_j \int_0^t \lambda_j^{1}(s)\: ds$ and $\lambda_0^k(t)=\min_{j} \lambda_j^k(t)$.
There exists $\delta=\delta(m,D,c_0, C_0, C_1, C_2, C_3,C_4)>0$ such that the following holds.
Define
$$\delta'=\delta\wedge \epsilon_2\wedge \int_0^{\epsilon_1} \lambda_0^1(s)\: ds \wedge \int_0^{\epsilon_1} \lambda_0^2(s)\: ds>0,$$
and set
\begin{equation*}
\delta_0:=
\begin{cases}
\gamma_0^{-1}(r\wedge \delta'), & \text{if }\gamma_0(\epsilon_1)\geq r\wedge \delta',\\
\epsilon_1,&\text{otherwise.}
\end{cases}
\end{equation*}
Then, $f_1(t,0)=f_2(t,0)$ for $t\in [0,\delta_0]$.
\end{thm} 

\section{Forward problems, inverse problems, and past work}\label{SectionForInv}
The results in this paper can be seen as studying a class of nonlinear
forward and inverse problems.
Indeed, suppose we have the same setup as described at the start of \cref{SectionResults}.
\newline\newline
\noindent\textbf{Forward Problem:}   Given $f_0(x)\in C([0,\epsilon_1];V)$
and $A_0(t)\in C^{2}([0,\epsilon_1];U)$ with $f_0(0)=A_0(0)$.
Let $\sM(t)$ be as in \cref{ThmResultExist}.
Suppose there exists $R(t)\in C^1([0,\epsilon_1];\GLm)$ with
$R(t)\sM(t)R(t)^{-1} = \diag{\lambda_1(t),\ldots, \lambda_m(t)}$,
and $\lambda_j(t)>0$, $\forall t$.  Let $f(t,x)$ be the solution to \cref{EqnResultMainEqn}
described in \cref{ThmResultExist}, with $f(0,x)=f_0(x)$, $f(t,0)=A_0(t)$.
The forward problem is the map:
\begin{equation*}
(f_0, A_0)\mapsto f(\epsilon_1,\cdot).
\end{equation*}
\newline
\noindent\textbf{Inverse Problem:}  The inverse problem is, given $f(\epsilon_1,\cdot)$
as described above, find $f_0$ and $A_0$.
Note that if $f(t,x)$ is the function described above, $\ft(t,x) = f(\epsilon_1-t,x)$
satisfies all the hypotheses of \cref{ThmResUnique} (here we assume
\cref{EqnResUniqueAddb}).
We have the following:
\begin{itemize}
\item The map $(f_0,A_0)\mapsto f(\epsilon_1,\cdot)$ is injective--\cref{ThmResUnique}.
\item The map $(f_0,A_0)\mapsto f(\epsilon_1,\cdot)$ is not surjective.
In fact, the only functions in the image of are Laplace transform type, modulo an appropriate error term--\cref{ThmResultExistCharaterize}.
\item The inverse map $f(\epsilon_1,\cdot)\mapsto (f_0,A_0)$ is unstable, but we
do have some stability results.  Indeed, if one only knows $f(\epsilon_1,x)$ up to error
terms of the form $O(e^{-r/x})$, then $f(\epsilon_1,\cdot)$ determins
$A_0(t)$ for $t\in [\delta_0-\epsilon_1,\epsilon_1]$, where $\delta_0$ is
described in \cref{ThmResUniqueStability}.
\item We have a proceedure to reconstruct $A_0(t)$ and $f_0(x)$ from
$f(\epsilon_1,x)$--\cref{RmkResUniqueReconstruct}.
\end{itemize}

The above class of inverse problems has, as special cases, some already well understood
inverse problems.  We next describe two of these.  For these problems,
we reverse time in the above discussion since we are focusing on the inverse problem.
In addition, the results in this paper are related to the famous Calder\'on problem,
and we describe this connection in \cref{AppendixCalderon}.

\subsection{The Laplace Transform}\label{SectionForInvLaplace}
As see in \cref{ExampleIntroExistence,ExampleIntroUniqueness} the Laplace transform
is closely related to the case $P(t,x,y,z)=y$ studied in this paper.
In fact, the following proposition makes this even more explicit.
For $a\in L^\infty([0,\infty))$ define the Laplace transform:
\begin{equation*}
\sL(a)(x) = \frac{1}{x}\int_0^{\infty} e^{-w/x}a(w)\: dw.
\end{equation*}

\begin{prop}\label{PropLinearDefiningDE}
Let $a\in C([0,\infty))\bigcap L^{\infty}([0,\infty))$.  For each $x>0$ there is a unique solution to the differential equation
\begin{equation}\label{EqnLinearDefiningDEQ}
\frac{\partial}{\partial t} f(t,x) = \frac{f(t,x)-a(t)}{x},
\end{equation}
such that $\sup_{t\geq 0} |f(t,x)|<\infty$.  For $t_0,t\geq 0$ define $a_{t_0}(t)=a(t_0+t)$.
This solution $f(t,x)$  is given by $f(t,x) = \sL(a_t)(x)$.  Furthermore, $f(t,x)$ extends to a continuous
function $f\in C([0,\infty)\times [0,\infty))$ by setting $f(t,0)=a(t)$.
\end{prop}
\begin{proof}
If we set
\begin{equation*}
f(t,x) = \sL(a_t)(x)=\frac{1}{x} \int_0^\infty e^{-s/x} a(t+s)\: ds = \frac{1}{x}\int_t^{\infty} e^{(t-s)/x} a(s)\: ds,
\end{equation*}
then it is clear that $f$ satisfies \cref{EqnLinearDefiningDEQ}, $\sup_{t\geq 0} |f(t,x)|<\infty$,
and that $f$ extends to a continuous function $f\in C([0,\infty)\times [0,\infty))$ by setting $f(t,0)=a(t)$.

Suppose $g(t,x)$ is another solution to \cref{EqnLinearDefiningDEQ} such that  $\sup_{t\geq 0} |g(t,x)|<\infty$.
Let $h=f-g$.  Then $h(t,x)$ satisfies $\frac{\partial}{\partial t} h(t,x) = h(t,x)/x$, $\sup_{t\geq 0} |h(t,x)|<\infty$.
This implies that $h(t,x) = e^{t/x} h(0,x)$ and we conclude $h(0,x)=0=h(t,x)$, for all $t$.  Thus $f(t,x)=g(t,x)$,
proving uniqueness.
\end{proof}

In light of \cref{PropLinearDefiningDE} one may define $\sL(a)$ (at least for
$a\in C([0,\infty))\bigcap L^{\infty}([0,\infty))$) in another way:
there is a unique $f(t,x)\in C([0,\infty)\times [0,\infty))$ with $\sup_{t\geq 0} |f(t,x)|<\infty$ and satisfying
\begin{equation*}
\dt f(t,x) = \frac{f(t,x)-f(t,0)}{x}, \quad f(t,0)=a(t).
\end{equation*}
$\sL(a)(x)$ is then defined to be $\sL(a)(x)=f(0,x)$.
Thus, the well known fact that $a\mapsto \sL(a)$ is injective follows
from uniqueness for the differential equation
\begin{equation*}
\dt f(t,x) = \frac{f(t,x)-f(t,0)}{x}.
\end{equation*}

\begin{example}
The above discussion leads naturally to the following ``nonlinear inverse Laplace transform''.  Indeed, let $P(y)$ be a polynomial in $y\in \R$.
Let $f_1(t,x),f_2(t,x)\in C([0,\epsilon_1]\times [0,\epsilon_0])$ satisfy, for $j=1,2$,
\begin{equation*}
\dt f_j(t,x) = \frac{P(f_j(t,x))-P(f_j(t,0))}{x}, \quad x\in (0,\epsilon_0].
\end{equation*}
Suppose:
\begin{itemize}
\item $f_1(0,x)=f_2(0,x)$, $\forall x\in [0,\epsilon_0]$.
\item $f_j(t,0)\in C^2([0,\epsilon_1])$, $j=1,2$.
\item $P'(f_j(t,0))>0$, for $t\in [0,\epsilon_1]$, $j=1,2$.
\end{itemize}
Then, by \cref{ThmResUnique}, $f_1(t,x)=f_2(t,x)$ for $(t,x)\in [0,\epsilon_1]\times [0,\epsilon_0]$.  When $P(y)=y$, this amounts to the inverse Laplace transform as discussed
above.
\end{example}

\subsection{Inverse Spectral Theory}\label{SectionForInvInvSpec}
In this section, we describe the results due to Simon in the influential work \cite{SimonANewApproachToInverseSpectalTheoryI},
where he gave a new approach to the theorem of Borg-Mar\v{c}enko that
the principal $m$-function for a finite interval or half-line Schr\"odinger operator determines
the potential.
 As we will show, this is closely related to the special case $P(t,x,y,z)=x^2y^2+y$ of
 the results studied in this paper.
 We will contrast our theorems and methods with those of Simon.

 Let $q\in L^1_{\mathrm{loc}}([0,\infty))$ with
 $\sup_{y>0} \int_y^{y+1} q(t)\vee 0\: dt<\infty$, and consider
 the Schr\"odinger operator $-\frac{d^2}{dt^2} +q(t)$.
 For each $z\in \C\setminus [\beta,\infty)$ (with $-\beta$ sufficiently large),
 there is a unique solution (up to multiplication by a constant)
 $u(\cdot,z)\in L^2([0,\infty))$ of $-\ddot{u}+qu=zu$.
 The principal $m$-function is defined by
 \begin{equation*}
 m(t,z)=\frac{\dot{u}(t,z)}{u(t,z)}.
 \end{equation*}
 It is a theorem of Borg \cite{BorgUniquensesTheoremsInTheSpectralTheory} and Mar\v{c}enko \cite{MarchenkoSomeQuestions} that $m(0,z)$ uniquely determines
 $q$--Simon \cite{SimonANewApproachToInverseSpectalTheoryI} saw this
 as an instance of uniqueness for a generalized differential equation,
 which we now explain in the framework of this paper.

 Indeed, it is easy to see that $m$ satisfies the Riccati equation
 \begin{equation}\label{EqnForInvRiccSimon}
 \dot{m}(t,z)=q(t)-z-m(t,z)^2,
 \end{equation}
 and well-known that $m$ has the asymptotics
 $m(t,-\kappa^2) = -\kappa - \frac{q(t)}{\kappa}+o(\kappa^{-1})$,
 as $\kappa \uparrow \infty$.  Thus, $q(t)$
 can be obtained from $m(t,\cdot)$ and \cref{EqnForInvRiccSimon}
 is a differential equation involving only $m$.
 Thus, if the equation \cref{EqnForInvRiccSimon} has uniqueness,
 then $m(0,z)$ uniquely determines $q(t)$.

 However, one does not need to full power of uniqueness for \cref{EqnForInvRiccSimon}.
 In fact, one needs only know uniqueness under the additional assumption
 that $m(t,z)$ is a principal $m$-function:  i.e., if $m_1(t,z)$ and $m_2(t,z)$ both
 satisfy \cref{EqnForInvRiccSimon} with $m_1(0,\cdot)=m_2(0,\cdot)$ and are both principal $m$-functions, then $m_1(t,z)=m_2(t,z)$, $\forall t,z$.  Simon proceeds
 via
 this weaker statement.

At this point, we rephrase these ideas into the language used in this paper.
For $x\geq 0$, $y\in \R$ define $P(x,y)=x^2y^2+y$.  Note that $P$ is
of the form covered in this paper (\cref{ExampleResultsPolynomial})
and $d_y P(0,y)=1$.  Given a principal $m$-function as above, define for $x\geq0$ small,
\begin{equation}\label{EqnForInvSimonmf}
f(t,x) :=
\begin{cases}
-\frac{1}{x} \left( m(t,-(2x)^{-2}) + (2x)^{-1} \right),&\text{if }x>0,\\
q(t), &\text{if }x=0.
\end{cases}
\end{equation}
It is easy to see from the above discussion that $f$ satisfies
\begin{equation}\label{EqnForInvSimonf}
\dt f(t,x) = \frac{P(x,f(t,x))-P(0,f(t,0))}{x}, \quad x>0.
\end{equation}
Furthermore, if $q$ is continuous then $f$ is continuous as well.
Thus to show $m(t,z)$ uniquely determines $q(t)$ it suffices to
show that \cref{EqnForInvSimonf} has uniqueness.

In this context, our results and the results of \cite{SimonANewApproachToInverseSpectalTheoryI} are closely related but have a
few differences:
\begin{itemize}
\item As discussed above, \cite{SimonANewApproachToInverseSpectalTheoryI}
 only considers solutions to \cref{EqnForInvRiccSimon} which are principal $m$-functions.  This forces $f(t,\cdot)$ in \cref{EqnForInvSimonmf} to be exactly of Laplace
transform type, $\forall t$.  As we have seen, not all solutions to
\cref{EqnForInvSimonf} are exactly of Laplace transform type.  In this way, our results are stronger than \cite{SimonANewApproachToInverseSpectalTheoryI} in that we prove uniqueness when the initial condition is not necessarily of Laplace transform type--we do not even require any sort of analyticity.\footnote{We learn a posteriori, in \cref{ThmResCharacterize}, that the initial condition must be of Laplace transform type modulo an error, but this is not assumed.}

\item We require $q\in C^2$, while \cite{SimonANewApproachToInverseSpectalTheoryI} requires no additional regularity on $q$.

\item The constant $\delta$ in \cref{ThmResCharacterize,ThmResUniqueStability}
is taken to be $\infty$ in \cite{SimonANewApproachToInverseSpectalTheoryI}.

\item Our results work for much more general polynomials than $P$.
\end{itemize}
The reason for the differences above is that, once $m$ is assumed to be a
principal $m$-function, one is able to use many theorems regarding Schr\"odinger equations to deduce the stronger results, which we did not obtain in our more
general setting.

That we assumed $q\in C^2$ is likely not essential.  For the specific case discussed in this
section, our methods do yield results for $q$ with lower regularity than $C^2$,
though we chose to not pursue this.
Moreover, even for the more general setting of our main results,
it seems likely that a more detailed study of the partial differential equations
which arise in this paper would lead to lower regularity requirements,
though this would require some new ideas.  That
$\delta$ is assumed small in  \cref{ThmResCharacterize,ThmResUniqueStability}
seems much more essential--this has to do with the fact that the
equations studied in this paper are non-linear in nature, unlike
the results in \cite{SimonANewApproachToInverseSpectalTheoryI}
which rested on the underlying linear theory of the Schr\"odinger equation.

\begin{rmk}
Many works followed \cite{SimonANewApproachToInverseSpectalTheoryI},
some of which dealt with $m$ taking values in square matrices;
e.g., \cite{GesztesySimonOnLocalBorgMarchenkoUniqunessResults}.
All of the above discussion applies to these cases as well.
\end{rmk}

\section{Convolution}
In this section, we record several results on the commutative and associative operation $\convt$ defined in \cref{EqnNotationConvt}.
In \cref{SectionConvPoly} we distill the consequences of these results into the form in which they will be used in the rest of the paper--and the reader
may wish to skip straight to those results on a first reading.
For this section, fix some $\epsilon>0$.

\begin{lemma}\label{LemmaConvOneDeriv}
Let $a\in C([0,\epsilon])$, $b\in C^1([0,\epsilon])$.  Then $\dw (a\convt b)(w)=a(w)b(0) + (a\convt b')(w)$.  In particular, if $b(0)=0$,
then $\dw (a\convt b)(w)= (a\convt b')(w)$.
\end{lemma}
\begin{proof}
This is immediate from the definitions.
\end{proof}


\begin{lemma}\label{LemmaConvAddSmoothing}
Let $l\geq -1$ and let $a\in C([0,\epsilon])$, $b\in C^{l+1}([0,\epsilon])$.  Suppose for $0\leq j\leq l-1$, $\dw[j] b(0)=0$.
Then $a\convt b\in C^{l+1}([0,\epsilon])$ and for $0\leq j\leq l$, $\dw[j] (a\convt b)(0)=0$.  Furthermore, if $a\in C^1([0,\epsilon])$,
then $a\convt b\in C^{l+2}([0,\epsilon])$.
\end{lemma}
\begin{proof}
By repeated applications of \cref{LemmaConvOneDeriv}, for $0\leq j\leq l$, $\dw[j] (a\convt b) = a\convt \dw[j]b$, and this expression clearly vanishes at $0$.
Applying \cref{LemmaConvOneDeriv} again, we see
$\dw[l+1] (a\convt b) =\dw (a\convt \dw[l] b) = a(w) \frac {\partial^{l} b}{\partial w^l}(0) + (a\convt \dw[l+1]b)$.
This expression is continuous, so $a\convt b\in C^{l+1}$.  Furthermore, if $a\in C^1$, it follows from one more application
of \cref{LemmaConvOneDeriv} that $\dw[l+2] (a\convt b) = \dw \left(a(w) \frac {\partial^{l} b}{\partial w^l}(0) + (a\convt \dw[l+1]b)\right)$
is continuous, and therefore $a\convt b\in C^{l+2}$.
\end{proof}

For the next few results, suppose $a_1,\ldots, a_L\in C^1([0,\epsilon])$ are given.  For $J=\{j_1,\ldots, j_k\}\subseteq \Lset$, we define
\begin{equation*}
\convts_{j\in J} a = a_{j_1}\convt \cdots \convt a_{j_k}.
\end{equation*}
With an abuse of notation, for $b\in C([0,\epsilon])$, we define $b\convt \left(\convts_{j\in \emptyset} a\right)=b$.

\begin{lemma}\label{LemmaConvnVanishes}
For each $n\in \Lset$, $a_1\convt \cdots \convt a_n\in C^n([0,\epsilon])$ and if $0\leq j\leq n-2$, $\dw[j]\left(a_1\convt\cdots \convt a_n\right)(0)=0$.
\end{lemma}
\begin{proof}
For $n=1$, the result is trivial.  We prove the result by induction on $n$, the base case being $n=2$ which follows from \cref{LemmaConvOneDeriv}.
We assume the result for $n-1$ and prove it for $n$.  By the inductive hypothesis, $a_1\convt \cdots \convt a_{n-1}\in C^{n-1}$ and vanishes to order
$n-3$ at $0$.  From here, the result follows from \cref{LemmaConvAddSmoothing} with $l=n-2$.
\end{proof}

Define
\begin{equation*}
I_L(a_1,\ldots, a_L):= \sum_{J\subsetneq \Lset} \left(\prod_{j\in J} a_j(0)\right) \left(\convts_{k\in J^c} a_k'\right),
\end{equation*}
and let $I_0=0$.

\begin{lemma}\label{LemmaConvDerivIL}
\begin{equation}\label{EqnConvILm1}
\dw[L-1] (a_1\convt \cdots \convt a_L) = \left( \prod_{j=1}^{L-1} a_j(0) \right) a_L + I_{L-1}(a_1,\ldots, a_{L-1}) \convt a_L.
\end{equation}
\begin{equation}\label{EqnConvIL}
\dw[L] (a_1\convt \cdots \convt a_L) = I_L(a_1,\ldots, a_L).
\end{equation}
\end{lemma}
\begin{proof}
We prove the result by induction on by induction on $L$.  The base case, $L=1$, is trivial.  We assume \cref{EqnConvILm1,EqnConvIL}
for $L-1$ and prove them for $L$.  We have, using repeated applications of \cref{LemmaConvOneDeriv,LemmaConvnVanishes},
\begin{equation*}
\begin{split}
&\dw[L-1] (a_1\convt \cdots \convt a_L) = \dw \left( \left(\dw[L-2](a_1\convt \cdots \convt a_{L-1}\right)\convt a_L \right)
\\&=\left(\dw[L-2] (a_1\convt \cdots \convt a_{L-1})\right)(0)a_L + \left(\dw[L-1] (a_1\convt\cdots \convt a_{L-1}\right)\convt a_L
\end{split}
\end{equation*}
Using our inductive hypothesis for \cref{EqnConvILm1} and the fact that $(b\convt c)(0)=0$ for any $b,c$,
\begin{equation*}
\begin{split}
\left(\dw[L-2] (a_1\convt \cdots \convt a_{L-1})\right)(0)a_L
=\left[\prod_{j=1}^{L-1} a_j(0)\right]a_L,
\end{split}
\end{equation*}
and using our inductive hypothesis for \cref{EqnConvIL},
\begin{equation*}
\left(\dw[L-1] (a_1\convt\cdots \convt a_{L-1})\right)\convt a_L = I_{L-1}(a_1,\ldots, a_{L-1})\convt a_L.
\end{equation*}
Combining the above equations yields \cref{EqnConvILm1}.
Taking $\dw$ of \cref{EqnConvILm1} and applying \cref{LemmaConvOneDeriv}, \cref{EqnConvIL} follows, completing the proof.
\end{proof}

\begin{cor}\label{CorConvDerivOfMon}
Let $A\in C^1([0,\epsilon];\R^m)$, $b\in C^1([0,\epsilon])$.  Then, for a multi-index $\alpha\in \N^m$,
\begin{equation*}
\dw[|\alpha|+1] \left(b\convt (\convt^{\alpha} A )\right) = \sum_{\substack{\beta\leq \alpha \\ |\beta|<|\alpha|}} \binom{\alpha}{\beta} b(0)A(0)^{\beta} \left(\convt^{\alpha-\beta} A'\right) + \sum_{\beta\leq \alpha} \binom{\alpha}{\beta} A(0)^{\beta} \left(b'\convt \left( \convt^{\alpha-\beta} A' \right)\right).
\end{equation*}
\end{cor}
\begin{proof}
This follows immediately from \cref{LemmaConvDerivIL}.
\end{proof}

\begin{lemma}\label{LemmaConvMultilinDiff}
Let $b_1,\ldots, b_L, c_1,\ldots, c_L\in C([0,\epsilon])$.  Then,
\begin{equation*}
b_1\convt \cdots \convt b_L - c_1\convt \cdots \convt c_L = \sum_{\emptyset \ne J\subseteq \Lset} (-1)^{|J|+1} \left(\convts_{l\in J} (b_l-c_l)\right)\convt\left(\convts_{l\not \in J} b_l\right)
\end{equation*}
\end{lemma}
\begin{proof}
This is standard, uses only the multilinearity of $\convt$, and can be proved using a simple induction.
\end{proof}

\begin{lemma}\label{LemmaConvLongdw}
Suppose $a_1,\ldots, a_L\in C^2([0,\epsilon])$.  Then,
\begin{equation*}
\begin{split}
&\dw[L] (a_1\convt\cdots\convt a_L) = \left(\prod_{l=1}^{L-1} a_l(0)\right) a_L' + \left( \sum_{l=1}^{L-1} \left( \prod_{\substack{1\leq k\leq L-1 \\ k\ne l } }a_k(0) \right) a_l'(0) \right)a_L
\\&+\left(  \sum_{l=1}^{L-1} \sum_{\substack{J\subsetneq  \Lmoset\\ l=\min J^c}} \left(\prod_{j\in J} a_j(0)\right) \left(    a_l'(0)\left( \convts_{\substack{k\in J^c \\ k\ne l}}a_k'\right) + a_l''\convt \left(  \convts_{\substack{ k\in J^c \\ k\ne l }} a_k' \right)  \right) \right) \convt a_L
\end{split}
\end{equation*}
\end{lemma}
\begin{proof}
Using \cref{LemmaConvOneDeriv,LemmaConvDerivIL}, we have
\begin{equation*}
\begin{split}
&\dw[L] (a_1\convt \cdots \convt a_L) = \dw\left( \left(\prod_{j=1}^{L-1} a_j(0)\right) a_L + I_{L-1}(a_1,\ldots, a_{L-1})\convt a_L \right)
\\&= \left(\prod_{j=1}^{L-1} a_j(0)\right) a_L' + I_{L-1}(a_1,\ldots, a_{L-1})(0) a_L + \left(\dw I_{L-1}(a_1,\ldots, a_{L-1})\right)\convt a_L
\end{split}
\end{equation*}
Since $(b\convt c)(0)=0$ for any $b,c$,
\begin{equation*}
I_{L-1}(a_1,\ldots, a_L)(0) = \sum_{l=1}^{L-1} \left(\prod_{\substack{1\leq k\leq L-1 \\ k\ne l }} a_k(0)\right)a_l'(0).
\end{equation*}
Using \cref{LemmaConvOneDeriv},
\begin{equation*}
\dw I_{L-1}(a_1,\ldots, a_{L-1}) = \sum_{l=1}^{L-1} \sum_{\substack{J\subsetneq \Lmoset \\ l=\min J^c }} \left(\prod_{j\in J} a_j(0)\right) \left( a_l'(0) \left(\convts_{\substack{k\in J^c \\ k\ne l } } a_k'\right) + a_l'' \convt\left( \convts_{\substack{ k\in J^c \\ k\ne l } } a_k'  \right) \right).
\end{equation*}
Combining the above equations yields the result.
\end{proof}

\begin{cor}
Let $a_1,\ldots, a_L\in C^2([0,\epsilon])$.
\begin{equation}\label{EqnConvDefnF1}
\dw[L] (a_1\convt \cdots \convt a_L) (w) = \left(\prod_{l=1}^{L-1} a_l(0)\right) a_L'(w) + F_1(w),
\end{equation}
\begin{equation}\label{EqnConvDefnF2}
\dw[L] (a_1\convt \cdots \convt a_L) (w) = F_2(w),
\end{equation}
where
\begin{equation*}
|F_1(w)|\lesssim \sup_{0\leq r\leq w} |a_L(r)|, \quad |F_2(w)|\lesssim \left(|a_{L-1}(0)| + \sup_{0\leq r\leq w} |a_L(r)|\right)\wedge \left( |a_L'(w)| + \sup_{0\leq r\leq w} |a_L(r)|\right),
\end{equation*}
where the implicit constants may depend on $L$, and upper bounds for $\epsilon$ and $\CjN{2}{a_j}$, $1\leq j\leq L$.
\end{cor}
\begin{proof}
The bound for $F_1$ follows immediately from \cref{LemmaConvLongdw}.  The bound for $F_2$ follows from \cref{EqnConvDefnF1} and the bound
for $F_1$.
\end{proof}

\begin{lemma}\label{LemmaConvProdQuotOne}
Let $a,b\in C^1([0,\epsilon])$.  Let $f(x)=\frac{1}{x}\int_0^{\epsilon} e^{-w/x} a(w)\: dw$ and $g(x)=\frac{1}{x}\int_0^\epsilon e^{-w/x} b(w)\: dw$.
Then, there exists $G\in C([0,\infty))$ such that
\begin{equation}\label{EqnConvMultiply}
f(x)g(x)=\frac{1}{x}\int_0^{\epsilon} e^{-w/x} \dw (a\convt b)(w)\: dw + \frac{1}{x}e^{-\epsilon/x} G(x).
\end{equation}
Also,
\begin{equation}\label{EqnConvDiffQuot}
\frac{f(x)-f(0)}{x} = \frac{1}{x}\int_0^{\epsilon} e^{-w/x} \frac{\partial a }{\partial w} (w) \: dw -\frac{1}{x}e^{-\epsilon/x} a(\epsilon).
\end{equation}
\end{lemma}
\begin{proof}
A straightforward computation shows
\begin{equation*}
f(x) g(x) = \frac{1}{x^2} \int_0^{\epsilon} e^{-u/x} \int_0^u a(w_1) b(u-w_1)\: dw_1 \: du + \frac{1}{x^2}\int_\epsilon^{2\epsilon} e^{-u/x}\int_{u-\epsilon}^{\epsilon} a(w_1) b(u-w_1)\: dw_1 \: du.
\end{equation*}
We have
\begin{equation*}
\begin{split}
&\frac{1}{x^2} \int_\epsilon^{2\epsilon} e^{-u/x} \int_{u-\epsilon}^{\epsilon} a(w_1)b(u-w_1)\: dw_1\: du
=\frac{1}{x^2}e^{-\epsilon/x} \int_0^\epsilon e^{-u/x} \int_{u}^{\epsilon} a(w_1) b(u+\epsilon-w_1)\: dw_1 \: du
=:\frac{1}{x} e^{-\epsilon/x} G_1(x),
\end{split}
\end{equation*}
where
 $G_1\in C([0,\epsilon])$.
 Also, using that $(a\convt b)(0)=0$,
 \begin{equation*}
 \begin{split}
 &\frac{1}{x^2} \int_0^{\epsilon} e^{-u/x} \int_0^u a(w_1) b(u-w_1)\:dw_1\: du= -\frac{1}{x}\int_0^\epsilon \left(\du e^{-u/x}\right) (a\convt b)(u)\: du
 \\&
 =-\frac{1}{x}e^{-\epsilon/x} (a\convt b)(\epsilon) + \frac{1}{x} \int_0^{\epsilon} e^{-u/x} \du (a\convt b)(u)\: du.
 \end{split}
 \end{equation*}
Combining the above equations yields \cref{EqnConvMultiply}.

We have,
\begin{equation*}
\frac{1}{x}\int_0^{\epsilon} e^{-w/x} \frac{\partial a}{\partial w} (w) \: dw = \frac{1}{x} e^{-\epsilon/x} a(\epsilon)-\frac{1}{x}a(0)+\frac{1}{x^2} \int_0^{\epsilon} e^{-w/x}a(w)\: dw = \frac{f(x)-f(0)}{x} +\frac{1}{x}e^{-\epsilon/x} a(\epsilon),
\end{equation*}
which proves \cref{EqnConvDiffQuot}.
\end{proof}

\begin{lemma}\label{LemmaConvProdQuotient}
Let $a_1,\ldots, a_n\in C^1([0,\epsilon])$.
Define $f_j(x)= \frac{1}{x}\int_0^{\epsilon} e^{-w/x} a_j(w)\: dw$.
Then, there are continuous functions $G_1, G_2\in C([0,\infty))$ such that
\begin{equation}\label{EqnConvToShowProd}
\prod_{j=1}^n f_j(x)  = \frac{1}{x}\int_0^{\epsilon} e^{-w/x} \dw[n-1] (a_1\convt\cdots \convt a_n)(w)\: dw + \frac{1}{x} e^{-\epsilon/x} G_1(x),
\end{equation}
\begin{equation}\label{EqnConvToShowProdQuot}
\frac{\prod_{j=1}^n f_j(x) - \prod_{j=1}^n f_j(0)}{x} = \frac{1}{x} \int_0^{\epsilon} e^{-w/x} \dw[n] (a_1\convt \cdots \convt a_n)(w)\: dw + \frac{1}{x^2} e^{-\epsilon/x}G_2(x).
\end{equation}
\end{lemma}
\begin{proof}
We prove \cref{EqnConvToShowProd} by induction on $n$.  $n=1$ is trivial
and $n=2$ is contained in \cref{LemmaConvProdQuotOne}.
We assume the result for $n-1$ and prove it for $n$.  Thus, we assume
\begin{equation}\label{EqnConvtProdInduct1}
\prod_{j=1}^{n-1} f_j(x) = \frac{1}{x}\int_0^{\epsilon} \dw[n-2] (a_1\convt \cdots \convt a_{n-1})(w)\: dw + \frac{1}{x} e^{-\epsilon/x}\Gt_1(x),
\end{equation}
where $\Gt_1\in C([0,\infty))$.
By \cref{LemmaConvnVanishes}, $a_1\convt \cdots \convt a_{n-1}\in C^{n-1}$ and vanishes to order $n-3$ at $0$.  Using this, and repeated applications \cref{LemmaConvOneDeriv}, we have
\begin{equation*}
\left(\dw[n-2] (a_1\convt \cdots \convt a_{n-1})\right)\convt a_n  = \dw[n-2] (a_1\convt \cdots \convt a_{n}).
\end{equation*}
Using this and \cref{LemmaConvProdQuotOne} we have, for some $\Gt_2\in C([0,\infty))$,
\begin{equation}\label{EqnConvtProdInduct2}
\begin{split}
&\left(  \frac{1}{x}\int_0^{\epsilon} \dw[n-2] (a_1\convt \cdots \convt a_{n-1})(w)\: dw \right) f_n(x)\\& = \frac{1}{x} \int_0^{\epsilon} e^{-w/x}\dw \left( \dw[n-2] (a_1\convt\cdots a_{n-1}) \convt a_n \right)(w)\: dw + \frac{1}{x} e^{-\epsilon/x} \Gt_2(x)
\\&=\frac{1}{x} \int_0^{\epsilon} \dw[n-1] (a_1\convt\cdots \convt a_n)(w) + \frac{1}{x}e^{-\epsilon/x} \Gt_2(x).
\end{split}
\end{equation}
Combining \cref{EqnConvtProdInduct1,EqnConvtProdInduct2}, we have
\begin{equation*}
\prod_{j=1}^n f_j(x) = \frac{1}{x} \int_0^{\epsilon} \dw[n-1] (a_1\convt\cdots \convt a_n)(w) + \frac{1}{x}e^{-\epsilon/x} \Gt_2(x) + \frac{1}{x} e^{-\epsilon/x} \Gt_1(x) f_n(x),
\end{equation*}
which proves \cref{EqnConvToShowProd}.

We turn to \cref{EqnConvToShowProdQuot}.  Using \cref{EqnConvDiffQuot,EqnConvToShowProd},
\begin{equation*}
\begin{split}
&\frac{\prod_{j=1}^n f_j(x) - \prod_{j=1}^n f_j(0)}{x} 
\\&= \frac{1}{x} \int_0^{\epsilon} e^{-w/x} \dw[n] (a_1\convt \cdots \convt a_n)(w)\: dw + \frac{1}{x^2}e^{-\epsilon/x} G_1(x) - \frac{1}{x}e^{-\epsilon/x} \dw[n-1]\bigg|_{w=\epsilon} (a_1\convt \cdots \convt a_n)(w).
\end{split}
\end{equation*}
Since $a_1\convt \cdots \convt a_n\in C^n$ (by \cref{LemmaConvnVanishes}), this completes the proof.
\end{proof}

\subsection{Smoothing}
The operation $\convt$ has smoothing properties, and this section is devoted to discussing the instances of these smoothing properties
which are used in this paper.  Fix $m\in \N$, $\epsilon_1,\epsilon_2>0$.

\begin{defn}\label{DefnSmoothingOperatation}
For $L\geq 0$, $n\geq 1$, and increasing functions $G_1,G_2,G_3:(0,\infty)\rightarrow (0,\infty)$, we say
\begin{equation*}
\sG:C^{0,L}([0,\epsilon_1]\times [0,\epsilon_2];\R^m)\rightarrow C^{0,L}([0,\epsilon_1]\times [0,\epsilon_2];\R^n)
\end{equation*}
is an $(L, G_1, G_2,G_3)$ operation if:
\begin{itemize}
\item $\sG(A)(t,w)$ depends only on the values of $A(t,r)$ for $r\in [0,w]$.  As a result, for $\delta\in (0,\epsilon_2]$, $\sG$ defines
a map
\begin{equation*}
\sG:C^{0,L}([0,\epsilon_1]\times [0,\delta];\R^m)\rightarrow C^{0,L}([0,\epsilon_1]\times [0,\delta];\R^n).
\end{equation*}

\item  For $0\leq k\leq L-1$, there are functions $\sG^k : C([0,\epsilon_1];\R^m)^{k+1}\rightarrow C([0,\epsilon_1];\R^m)$
such that
\begin{equation*}
\sG^k:C^L([0,\epsilon_1];\R^m)\times C^{L-1}([0,\epsilon_1];\R^m)\times \cdots \times C^{L-1}([0,\epsilon_1];\R^m)\rightarrow C^{L-k-1}([0,\epsilon_1];\R^n),
\end{equation*}
and
\begin{equation*}
\dw[k] \sG(A)(t,w)\bigg|_{w=0} = \sG^k\left( A(\cdot,0), \frac{\partial A}{\partial w} (\cdot,0),\ldots, \frac{\partial^k A}{\partial w^k}(\cdot, 0) \right)(t).
\end{equation*}

\item The following hold $\forall M\in (0,\infty)$, $\delta\in (0,\epsilon_2]$.
\begin{itemize}
\item $\forall A\in C^{0,L}([0,\epsilon_1]\times [0,\delta];\R^m)$ with $\CjN{0,L}{A}\leq M$, $\CjN{0,L}{\sG(A)}\leq G_1(M)$.
\item $\forall A, B\in C^{0,L}([0,\epsilon_1]\times[0,\delta];\R^m)$ with $\CjN{0,L}{A}, \CjN{0,L}{B}\leq M$,
\begin{equation*}
\CjN{0,L}{\sG(A)-\sG(B)}\leq G_2(M) \CjN{0,L}{A-B}.
\end{equation*}
\item For $0\leq k\leq L-1$, and $g_1,\ldots, g_k$ with $g_j\in C^{L-j}([0,\epsilon_1];\R^m)$ and $\CjN{L-j}{g_j}\leq M$,
\begin{equation*}
\CjN{L-k-1}{\sG^k(g_1,\ldots, g_k)}\leq G_3(M).
\end{equation*}
\end{itemize}
\end{itemize}
\end{defn}

Below we use $\convt$ to construct several examples, in the case $n=1$, of $(2,G_1,G_2,G_3)$ operations.

\begin{lemma}\label{LemmaSmoothing1}
Let $\alpha\in \N^m$ be a multi-index with $|\alpha|\geq 2$, and let $b(t,w)\in C([0,\epsilon_1]\times [0,\epsilon_2])$.
For $A\in C^{0,2}([0,\epsilon_1]\times [0,\epsilon_2];\R^m)$ set
$$\sG(A)(t,w):=(b(t,\cdot)\convt \left(\convt^{\alpha} A'(t,\cdot) ) \right)(w).$$
Then, $\sG$ is a $(2,G_1,G_2,G_3)$ operation, where the functions $G_1$, $G_2$, and $G_3$ can be chosen to depend only on
$\alpha$, $m$, and upper bounds for $\epsilon_2$ and $\CjN{0}{b}$.
\end{lemma}
\begin{proof}
Let $k_1=\min\{l : \alpha_l\ne 0\}$ and $k_2 =\min\{l:(\alpha-e_{k_1})_l\ne 0\}$.
Using \cref{LemmaConvOneDeriv}, we have
\begin{equation*}
\dw \sG(A)(t,w) = A_{k_1}'(t,0) \left( b(t,\cdot)\convt \left(  \convt^{\alpha-e_k} A'(t,\cdot)\right) \right)(w) + \left(b(t,\cdot) \convt A_{k_1}''(t,\cdot)\convt \left( \convt^{\alpha-e_{k_1}} A'(t,\cdot) \right)\right)(w),
\end{equation*}
and
\begin{equation*}
\begin{split}
\dw[2] \sG(A)(t,w) &= A_{k_1}'(t,0) A_{k_2}'(t,0) \left( b(t,\cdot)\convt \left( \convt^{\alpha-e_{k_1}-e_{k_2}} A'(t,\cdot) \right) \right)(w)
\\&+ A_{k_1}'(t,0) \left( b(t,\cdot) \convt A_{k_2}''(t,\cdot)\convt \left( \convt^{\alpha-e_{k_1}-e_{k_2}} A'(t,\cdot) \right) \right)(w)
\\&+A_{k_2}'(t,0) \left(b(t,\cdot) \convt A_{k_1}''(t,\cdot) \convt \left( \convt^{\alpha-e_{k_1}-e_{k_2}} A'(t,\cdot) \right)\right)(w)
\\&+ \left( b(t,\cdot)\convt A_{k_1}''(t,\cdot)\convt A_{k_2}''(t,\cdot) \convt \left( \convt^{\alpha-e_{k_1}-e_{k_2}} A'(t,\cdot)\right) \right)(w).
\end{split}
\end{equation*}
For any $c_1, c_2$, we have $(c_1\convt c_2)(0)=0$, we may therefore take $\sG^0=0$ and $\sG^1=0$.
Using the above formulas, combined with \cref{LemmaConvMultilinDiff}, the result follows.
\end{proof}

\begin{lemma}\label{LemmaSmoothing2}
Suppose $|\alpha|=1$ and $b(t,w)\in C^{0,1}([0,\epsilon_1]\times [0,\epsilon_2])$.  For $A\in C^{0,2}([0,\epsilon_1]\times [0,\epsilon_2];\R^m)$
set $$\sG(A)(t,w):= \left(b(t,\cdot)\convt \left(\convt^{\alpha} A'(t,\cdot)\right)\right)(w).$$  Then $\sG$ is a $(2,G_1,G_2,G_3)$ operation,
where the functions $G_1$, $G_2$, and $G_3$ can be chosen to depend only on $m$ and upper bounds for $\epsilon_2$ and $\CjN{0,1}{b}$.
\end{lemma}
\begin{proof}
Without loss of generality we take $\alpha=e_1$, so that $\sG(A)(t,w)=(b(t,\cdot)\convt A_1'(t,\cdot))(w)$.
Using \cref{LemmaConvOneDeriv} we have
\begin{equation*}
\dw \sG(A)(t,w) = A_1'(t,0) b(t,w) + (b(t,\cdot)\convt A_1''(t,\cdot))(w),
\end{equation*}
\begin{equation*}
\dw[2] \sG(A)(t,w) = A_1'(t,w) b'(t,w) + b(t,0) A_1''(t,w) + (b'(t,\cdot)\convt A_1''(t,\cdot))(w).
\end{equation*}
In particular,
\begin{equation*}
\sG(A)(t,0)=0, \quad \dw\bigg|_{w=0} \sG(A)(t,w) = A_1'(t,0) b(t,0).
\end{equation*}
Using the above formulas, the result follows easily.
\end{proof}

\begin{lemma}\label{LemmaSmoothing3}
Suppose $|\alpha|\geq 2$ and $b(t)\in C([0,\epsilon_1])$.
For $A\in C^{0,2}([0,\epsilon_1]\times [0,\epsilon_2];\R^m)$ set
\begin{equation*}
\sG(A)(t,w):=b(t)\left(\convt^{\alpha} A'(t,\cdot)\right)(w).
\end{equation*}
Then, $\sG$ is a $(2,G_1,G_2,G_3)$ operation, where the functions $G_1$, $G_2$, and
$G_3$ can be chosen to depend only on $m$ and upper bounds for $\epsilon_2$
and $\CjN{0}{b}$.
\end{lemma}
\begin{proof}
Let $k_1=\min\{l : \alpha_l \ne 0\}$ and $k_2=\min\{ l : (\alpha-e_{k_1})_l\ne 0\}$.
Using \cref{LemmaConvOneDeriv} we have
\begin{equation*}
\dw \sG(A)(t,w) = b(t) A_{k_1}'(t,0) \left(\convt^{\alpha-e_{k_1}} A'(t,\cdot)\right)(w)
+b(t) \left(A_{k_1}''(t,\cdot)\convt \left( \convt^{\alpha-e_{k_1}} A'(t,\cdot) \right)\right)(w),
\end{equation*}
and
\begin{equation*}
\begin{split}
\dw[2] \sG(A)(t,w) = & b(t) A_{k_1}'(t,0) A_{k_2}'(t,0) \left(\convt^{\alpha-e_{k_1}-e_{k_2}} A'(t,\cdot)\right)(w)
\\&+ b(t) A_{k_1}'(t,0) \left(A_{k_2}''(t,\cdot) \convt\left(\convt^{\alpha-e_{k_1}-e_{k_2}} A'(t,\cdot)\right)\right)(w)
\\&+ b(t) A_{k_2}'(t,0) \left(A_{k_1}''(t,\cdot) \convt\left(\convt^{\alpha-e_{k_1}-e_{k_2}} A'(t,\cdot)\right)\right)(w)
\\& + b(t) \left(A_{k_1}''(t,\cdot) \convt A_{k_2}''(t,\cdot) \convt \left( \convt^{\alpha-e_{k_1}-e_{k_2}} A'(t,\cdot) \right)\right)(w).
\end{split}
\end{equation*}
In particular,
\begin{equation*}
\sG(A)(t,0)=0,\quad \dw\bigg|_{w=0} \sG(A)(t,w) =
\begin{cases}
0, &\text{if }|\alpha|>2,\\
b(t)A_{k_1}'(t,0) A_{k_2}'(t,0), &\text{if }|\alpha|=2.
\end{cases}
\end{equation*}
Using the above formulas, combined with \cref{LemmaConvMultilinDiff}, the result follows easily.
\end{proof}

\begin{lemma}\label{LemmaSmoothing4}
Suppose $d\in C^{0,2}([0,\epsilon_1]\times [0,\epsilon_2])$ is such that
$d(t,0)\in C^1([0,\epsilon_1])$.  For $A\in C^{0,2}([0,\epsilon_1]\times [0,\epsilon_2];\R^m)$ set
\begin{equation*}
\sG(A)(t,w):=d(t,w).
\end{equation*}
Then, $\sG$ is a $(2, G_1, G_2,G_3)$ operation, where the functions $G_1$, $G_2$,
and $G_3$ can be chosen to depend only on upper bounds for
$\CjN{0,2}{d}$ and $\CjN{1}{d(\cdot,0)}$.
\end{lemma}
\begin{proof}
This follows immediately from the definitions.
\end{proof}

\begin{lemma}\label{LemmaSmoothing5}
Suppose $\sG:C^{0,L}([0,\epsilon_1]\times [0,\epsilon_2];\R^m)\rightarrow C^{0,L}([0,\epsilon_1]\times [0,\epsilon_2])$ is an
$(L,G_1,G_2,G_3)$ operation.  Let $\beta\in \N^{m}$ be a multi-index,
and define
\begin{equation*}
\sGt(A)(t,w):= A(t,0)^{\beta} \sG(A)(t,w).
\end{equation*}
Then, $\sGt$ is an $(L, \Gt_1, \Gt_2,\Gt_3)$ operation, where $\Gt_1$, $\Gt_2$, and $\Gt_3$ can be chosen to depend only on $G_1$, $G_2$, $G_3$, $L$, and $\beta$.
\end{lemma}
\begin{proof}
This follows immediately from the definitions.
\end{proof} 

\subsection{Polynomials}\label{SectionConvPoly}
For this section, we take all the same notation and assumptions as in the beginning
of \cref{SectionResults}.  Thus, we have $b_{\alpha,j}$, $c_{\alpha,j}$,
$P(t,x,y,z)$, and $\Ph(t, A(\cdot),z)(w)$ as described in that section.

\begin{lemma}\label{LemmaConvPolyDifferenceEqn}
Let $\delta\in (0,\epsilon_2]$ and $A(t,w)\in C^{0,1}([0,\epsilon_1]\times [0,\delta];\R^m)$
with $A(t,0)\in C([0,\epsilon_1];U)$. Define $f(t,x)\in C([0,\epsilon_1]\times [0,\epsilon_0];\R^m)$ by
$f(t,x) = \frac{1}{x}\int_0^\delta e^{-w/x} A(t,w)\: dw$.  Then,
\begin{equation*}
\frac{P(t,x,f(t,x),f(t,0))- P(t,0,f(t,0),f(t,0)) }{x}  = \frac{1}{x} \int_0^\delta e^{-w/x} \Ph(t,A(t,\cdot),A(t,0))(w)\: dw + \frac{1}{x^2} e^{-\delta/x} G(t,x),
\end{equation*}
where $G(t,x)\in C([0,\epsilon_1]\times [0,\epsilon_0];\R^m)$.
\end{lemma}
\begin{proof}
This follows from \cref{LemmaConvProdQuotient}, using the fact that $f(t,0)=A(t,0)$.
\end{proof}

\begin{prop}\label{PropConvPolyIsOperation}
Let $\delta\in (0,\epsilon_2]$.
For $A\in C^{0,2}([0,\epsilon_1]\times [0,\delta];\R^m)$ and $A_0(t)\in C^1([0,\epsilon_1];\R^m)$,
\begin{equation*}
\Ph(t, A(t,\cdot), A_0(t))(w) = d_y P(t,0,A(t,0),A_0(t))A'(t,w)+ \sG_{A_0}(A)(t,w),
\end{equation*}
where $\sG_{A_0}:C^{0,2}([0,\epsilon_1]\times [0,\epsilon_2];\R^m)\rightarrow C^{0,2}([0,\epsilon_1]\times [0,\epsilon_2];\R^m)$ is a $(2,G_1,G_2,G_3)$
operation\footnote{See \cref{DefnSmoothingOperatation} for the definition of a $(2,G_1,G_2,G_3)$ operation.}.  The functions $G_1$, $G_2$, and $G_3$ can be chosen to depend
only on $C_0$, $m$, $D$,\footnote{See \cref{SectionResults} for the definitions of these
various constants.} and upper bounds for $\epsilon_2$ and $\CjN{1}{A_0}$.
\end{prop}
\begin{proof}
By linearity, it suffices to prove the result for $P$ a monomial in $y$.  I.e.,
\begin{equation*}
P(t,x,y,z) = c_{\alpha,j}(t,x,z) y^{\alpha} e_j,
\end{equation*}
for some $j\in \mset$, $\alpha\in \N^m$ with $|\alpha|\leq D$.
In this case,
\begin{equation}\label{EqnConvPolyPhMonom}
\Ph(t, A(t,\cdot),z)(w) = \dw[|\alpha|+1]\left (b_{\alpha,j}(t,\cdot,z) \convt \left(\convt^{\alpha} A(t,\cdot)\right)\right)(w)e_j.
\end{equation}
Using \cref{CorConvDerivOfMon} and the fact that $b_{\alpha,j}(t,0,z)=c_{\alpha,j}(t,0,z)$,
\begin{equation}\label{EqnConvPolyDerivMon}
\begin{split}
\Ph(t,A(t,\cdot),A_0(t))(w) &= \sum_{l=1}^m \alpha_l b_{\alpha,j}(t,0,A_0(t)) A(t,0)^{\alpha-e_l} A_l'(t,w) e_j
\\&+\sum_{\substack{\beta\leq \alpha \\ |\beta|<|\alpha|-1}} \binom{\alpha}{\beta} b_{\alpha,j}(t,0,A_0(t)) A(t,0)^{\beta} \left( \convt^{\alpha-\beta} A'(t,\cdot) \right)(w) e_j
\\& + \sum_{\beta\leq \alpha} \binom{\alpha}{\beta} A(t,0)^\beta \left( b_{\alpha,j}'(t,\cdot, A_0(t))\convt \left(\convt^{\alpha-\beta} A'(t,\cdot)\right) \right)(w) e_j
\end{split}
\end{equation}
Note,
$$\sum_{l=1}^m \alpha_l b_{\alpha,j}(t,0,A_0(t)) A(t,0)^{\alpha-e_l} A_l'(t,w) e_j = d_y P(t,0,A(t,0), A_0(t))A'(t,w).$$
Thus, it remains to show the final two terms
on the right hand side of \cref{EqnConvPolyDerivMon} are a $(2,G_1,G_2,G_3)$
operation.  This follows from \cref{LemmaSmoothing1,LemmaSmoothing2,LemmaSmoothing3,LemmaSmoothing4,LemmaSmoothing5},
completing the proof.
\end{proof}

\begin{prop}\label{PropConvPolyDifferencePh}
In addition to the other assumptions of this section, we assume \cref{EqnResUniqueAddb}.
Let $\delta\in (0,\epsilon_2]$ and let $A,B\in C^{0,2}([0,\epsilon_1]\times [0,\delta];\R^m)$.
Set $g(t,w)=A(t,w)-B(t,w)$.  Then,
\begin{equation*}
\Ph(t, A(t,\cdot), A(t,0))(w)-\Ph(t,B(t,\cdot),B(t,0))(w) = d_y P (t,0,A(t,0),A(t,0))g'(t,w)+F(t,w),
\end{equation*}
where there exists a constant $C$ with
\begin{equation*}
|F(t,w)|\leq C \sup_{0\leq r\leq w} |g(t,r)|, \quad \forall t,w.
\end{equation*}
Here, $C$ is allowed to depend on any of the ingredients in the proposition, including $A$ and $B$.
\end{prop}
\begin{proof}
By linearity, it suffices to prove the result for $P$ a monomial in $y$.  I.e.,
\begin{equation*}
P(t,x,y,z)= c_{\alpha,j}(t,x,z) y^{\alpha} e_j,
\end{equation*}
for some $j\in \mset$ and $\alpha\in \N^m$ with $|\alpha|\leq D$.  In this case
$\Ph$ is given by \cref{EqnConvPolyPhMonom}.
Using \cref{LemmaConvMultilinDiff},
\begin{equation*}
\begin{split}
&\Ph(t,A(t,\cdot),A(t,0))(w) - \Ph(t, B(t,\cdot), B(t,0))
 = \sum_{l=1}^m \alpha_l \dw[|\alpha|+1] \left( b_{\alpha,j}(t,\cdot,A(t,0)) \convt g_l(t,\cdot) \convt \left( \convt^{\alpha-e_l} A(t,\cdot) \right) \right)(w)e_j
 \\&+\sum_{\substack{ \beta\leq \alpha \\ |\beta|\geq 2}} (-1)^{|\beta|+1} \binom{\alpha}{\beta} \dw[|\alpha|+1] \left( b(t,\cdot, A(t,0))\convt \left(\convt^{\beta} g(t,\cdot)\right) \convt \left( \convt^{\alpha-\beta} A(t,\cdot) \right) \right)(w)e_j
 \\&+\dw[|\alpha|+1] \left( \left(b_{\alpha,j}(t,\cdot,A(t,0)) - b_{\alpha,j}(t,\cdot,B(t,0))\right) \convt \left(\convt^{\alpha} B(t,\cdot)\right)
 \right)(w)e_j
 \\&=:(I)+(II)+(III).
\end{split}
\end{equation*}
We study the three terms on the right hand side of the above equation separately.
Applying \cref{EqnConvDefnF1} to each term of the sum in $(I)$, with $g_l$
playing the role of $a_L$, and using the fact that $b_{\alpha,j}(t,0,z)=c_{\alpha,j}(t,0,z)$,
\begin{equation*}
(I) = \sum_{l=1}^m \alpha_l b_{\alpha,j}(t,0,A(t,0)) A(t,0)^{\alpha-e_l} g_l'(t,0)e_j+ F_1(t,w)
=d_y P(t,0, A(t,0),A(t,0)) g'(t,0) +F_1(t,w),
\end{equation*}
where $|F_1(t,w)|\lesssim \sup_{0\leq r\leq w} |g(t,r)|$.
Turning to $(II)$, we note that in each term in the sum defining $(II)$, $|\beta|\geq 2$,
and so there are at least two coordinates (counting repetitions) of $g(t,\cdot)$
in the convolution.  Applying \cref{EqnConvDefnF2} to each term of the sum,
with these two coordinates of $g(t,\cdot)$ playing the roles of $a_L$ and
$a_{L-1}$, we see $(II)=F_2(t,w)$ where
$|F_2(t,w)|\lesssim \sup_{0\leq r\leq w} |g(t,r)|$.
Finally, for $(III)$, we use that as $t$ varies over $[0,\epsilon_1]$, $A(t,0)$ and $B(t,0)$
range over a compact subset of $U$.  Applying \cref{EqnConvDefnF2}
with $b_{\alpha,j}(t,\cdot,A(t,0))-b_{\alpha,j}(t,\cdot, B(t,0))$ playing the
role of $a_L$, and using \cref{EqnResUniqueAddb}, we see $(III)=F_3(t,w)$, where
\begin{equation*}
\begin{split}
|F_3(t,w)|&\lesssim \sup_{0\leq r\leq w} |b_{\alpha,j}(t,r,A(t,0))- b_{\alpha,j}(t,r,B(t,0))|
+ | b_{\alpha,j}'(t,w,A(t,0))-b_{\alpha,j}'(t,w,B(t,0))|
\\&\lesssim |g(t,0)|\lesssim \sup_{0\leq r\leq w} |g(t,r)|.
\end{split}
\end{equation*}
Summing the above three estimates completes the proof.
\end{proof} 

\section{Ordinary Differential Equations}
In this section, we prove some auxiliary results concerning ODEs which are needed in the remainder of the paper. 

\subsection{Chronological Calculus}
Let $m\in \N$ and let $J=[a,b]$ for some $a<b$.
Let $M(t):J\rightarrow \M^{m\times m}$ be locally bounded and measurable.

\begin{defn}
For $t\in J$, we define $\Texp{\int_a^t A(s)\: ds}$ by $\Texp{\int_a^t A(s)\: ds} = E(t)$
is the unique solution $E:J\rightarrow \M^{m\times m}$ to the differential equation
\begin{equation*}
\dot{E}(t) = A(t) E(t), \quad E(a)=I,
\end{equation*}
where $I$ denotes the $m\times m$ identity matrix.
\end{defn}

For the rest of this section, fix some $\epsilon_0>0$.

\begin{prop}\label{PropChronMainTexp}
Let $\sM(t,x)\in C(J\times [0,\epsilon_0]; \M^{m\times m})$ be such that
there exists $R(t)\in C^1(J; \GLm)$ with $R(t)\sM(t,0)R(t)^{-1} = \diag{\lambda_1(t),\ldots, \lambda_m(t)}$, with $\lambda_j(t)>0$ for all $t$.  Set $\lambda_0(t)=\min_{1\leq j\leq m} \lambda_j(t)$.  Then,
$\forall \delta\in [0,1)$, $\exists x_0\in (0,\epsilon_0]$, $\forall x\in (0,x_0]$, $\forall t\in J$,
\begin{equation*}
\left\|  \Texp{ -\frac{1}{x}\int_a^t \sM(s,x)\: ds } \right\|\leq \left\| R(t)^{-1} \right\| \left\|R(a)\right\| \exp\left( -\frac{\delta}{x} \int_a^t \lambda_0(s)\: ds \right).
\end{equation*}
\end{prop} 

To prove \cref{PropChronMainTexp}, we introduce a lemma.

\begin{lemma}\label{LemmaChronTexp}
Let $\sM(t,x)\in C(J\times [0,\epsilon_0];\M^{m\times m})$ and set
$2\lambda_0(t)$ to be the least eigenvalue of $\sM(t,0)^{\top}+\sM(t,0)$.
We assume $\lambda_0(t)>0$, $\forall t\in J$.  Then, $\forall \delta\in [0,1)$,
$\exists x_0\in (0,\epsilon_0]$, $\forall x\in (0,x_0]$,
\begin{equation*}
\left\|  \Texp{ -\frac{1}{x}\int_a^t \sM(s,x)\: ds } \right\|\leq \exp\left( -\frac{\delta}{x} \int_a^t \lambda_0(s)\: ds \right).
\end{equation*}
\end{lemma}
\begin{proof}
Let $\sN(t,x)=\sM(t,x)-\sM(t,0)$ so that $\sN(t,x)\in C(J\times[0,\epsilon_0];\M^{m\times m})$ and $\sN(t,0)=0$.  Fix $\delta\in [0,1)$ and take $x_0\in (0,\epsilon_0]$
 so small $\forall (t,x)\in J\times[0,x_0]$, $\| \sN(t,x) \|\leq \inf_{s\in J} (1-\delta)\lambda_0(s)$.
 
 Let $\theta_0\in \R^m$ and set $\theta(t,x):=\Texp{-\frac{1}{x}\int_a^t \sM(s,x)\: ds}\theta_0$.  Then,
 \begin{equation*}
 \begin{split}
 &\dt |\theta(t,x)|^2 = -\ip{\theta(t,x)}{ \left( \frac{1}{x} \sM(t,x)^{\top} + \frac{1}{x} \sM(t,x) \right)\theta(t,x) }
 \\&=-\frac{1}{x} \ip{\theta(t,x)}{\left(\sM(t,0)^{\top} + \sM(t,0)\right)\theta(t,x)} - \frac{1}{x}\ip{\theta(t,x)}{\left(\sN(t,x)^{\top}+\sN(t,x)\right)\theta(t,x)}
 \\&\leq -\frac{2}{x} \lambda_0(t) |\theta(t,x)|^2 + \frac{2}{x} \|\sN(t,x)\| |\theta(t,x)|
 \leq -\frac{2}{x} \delta \lambda_0(t) |\theta(t,x)|^2.
 \end{split}
 \end{equation*}
 By Gr\"onwall's inequality, we have $|\theta(t,x)|^2 \leq |\theta_0|^2 \exp\left(-\frac{2\delta}{x} \int_a^t \lambda_0(s)\: ds\right)$.  Taking square roots yields the result.
\end{proof}

\begin{proof}[Proof of \cref{PropChronMainTexp}]
Let $\Lambda(t)=\diag{\lambda_1(t),\ldots, \lambda_m(t)}=R(t)\sM(t,0)R(t)^{-1}$.
For $\theta_0\in \R^m$, set $\theta(t,x)=\Texp{-\frac{1}{x}\int_a^t \sM(s,x)\: ds}\theta_0$.
Let $\gamma(t,x)=R(t)\theta(t,x)$.  $\gamma$ satisfies
\begin{equation*}
\dt \gamma(t,x) = -\frac{1}{x} R(t) \sM(t,x)R(t)^{-1} \gamma(t,x)+\dot{R}(t)R(t)^{-1}\gamma(t,x)= -\frac{1}{x} \sMt(t,x)\gamma(t,x),
\end{equation*}
where $\sMt(t,x) = \Lambda(t) + R(t)\left(\sM(t,x)-\sM(t,0)\right)R(t)^{-1} - x \dot{R}(t)R(t)^{-1}$.
In particular, note $\sM(t,0)=\Lambda(t)$.
It follows that $\gamma(t,x)= \Texp{-\frac{1}{x} \sMt(s,x)\: ds}\gamma(a,x)$.

Fix $\delta\in [0,1)$.  By \cref{LemmaChronTexp}, $\exists x_0\in (0,\epsilon_0]$ (independent of $\theta_0$) such that for $x\in (0,x_0]$, 
\begin{equation*}
\left\| \Texp{-\frac{1}{x} \int_a^t \sMt(s,x)\: ds}  \right\|\leq \exp\left( -\frac{\delta}{x}\int_a^t \lambda_0(s)\: ds \right).
\end{equation*}
Hence, for $x\in (0,x_0]$,
\begin{equation*}
\begin{split}
&|\theta(t,x)| \leq \left\| R(t)^{-1}\right\| |\gamma(t,x)|
\leq \left\| R(t)^{-1}  \right\|\exp\left( -\frac{\delta}{x} \int_a^t \lambda_0(s)\: ds\right)|\gamma(a,x)|
\\&\leq \left\| R(t)^{-1}  \right\| \left\|R(a)\right\|\exp\left( -\frac{\delta}{x} \int_a^t \lambda_0(s)\: ds\right)|\theta_0|.
\end{split}
\end{equation*}
The result follows.
\end{proof}

\subsection{A Basic Existence Result}
Fix $\epsilon_1,\epsilon_0>0$ and let $W\subseteq \R^m$ be an open neighborhood of $0\in \R^m$.
Suppose $\sM(t,x,y)\in C([0,\epsilon_1]\times [0,\epsilon_0]\times W;\M^{m\times m})$
be such that for every compact set $K\Subset W$,
\begin{equation*}
\sup_{\substack{t\in [0,\epsilon_1], x\in [0,\epsilon_0] \\ y_1,y_2\in K, y_1\ne y_2}} \frac{\left\|  \sM(t,x,y_1)-\sM(t,x,y_2) \right\|}{|y_1-y_2|}<\infty.
\end{equation*}
Let $G(t,x,y)\in C([0,\epsilon_1]\times [0,\epsilon_0]\times W; \R^m)$ be such that  for every compact set
$K\Subset W$,
\begin{equation*}
\sup_{\substack{t\in [0,\epsilon_1], x\in [0,\epsilon_0] \\ y_1,y_2\in K, y_1\ne y_2}} \frac{\left|  G(t,x,y_1)-G(t,x,y_2) \right|}{|y_1-y_2|}<\infty.
\end{equation*}
Let $g_0\in C([0,\epsilon_0];\R^m)$ have $g_0(0)=0$.
The goal of this section is to study the differential equation
\begin{equation}\label{EqnPerturbMainEqn}
\dt g(t,w) = -\frac{1}{x} \sM(t,x,g(t,x))g(t,x) + G(t,x,g(t,x)), \quad x>0
\end{equation}
with the initial condition $g(0,x)=g_0(x)$.  The main result is the following.

\begin{prop}\label{PropPertubMainProp}
Set $\sM_0(t)=\sM(t,0,0)$.  We suppose that there exists $R(t)\in C^1([0,\epsilon_1];\GLm)$ such that
$$R(t) \sM_0(t) R(t)^{-1}  = \diag{\lambda_1(t),\ldots, \lambda_m(t)}$$
and $\lambda_j(t)>0$, $\forall j,t$.
Then, there exists $\delta_0\in (0,\epsilon_0]$ and a function
$g(t,x)\in C([0,\epsilon_1]\times [0,\delta_0];W)$ such that
$g(0,x)=g_0(x)$ $\forall x\in [0,\delta_0]$, $g(t,0)=0$ $\forall t\in [0,\epsilon_1]$, and
$g$ satisfies \cref{EqnPerturbMainEqn}.
\end{prop}

To prove \cref{PropPertubMainProp}, we need two lemmas.
As in \cref{PropPertubMainProp}, set $\sM_0(t)=\sM(t,0,0)$.  For these lemmas, instead
of assuming the existence of $R(t)$ as in \cref{PropPertubMainProp}, we
let $2\lambda_0(t)$ be the least eigenvalue of $\sM_0(t)^{\top}+\sM_0(t)$
and we assume $\lambda_0(t)>0$, $\forall t\in [0,\epsilon_1]$.

\begin{lemma}\label{LemmaPerturbSymmFixed}
Under the the assumption $\lambda_0(t)>0$ $\forall t$, the following holds.
For all $\epsilon>0$, there exists $\delta>0$ such that for all $x_0\in (0,\delta]$,
there exists a unique solution $g^{x_0}(t)\in C^1([0,\epsilon_1];B^{m}(\epsilon)\cap W)$ to the differential
equation
\begin{equation*}
\dt g^{x_0}(t) = -\frac{1}{x_0} \sM(t,x_0, g^{x_0}(t))g^{x_0}(t) + G(t,x_0,g^{x_0}(t)), \quad g^{x_0}(0)=g_0(x_0).
\end{equation*}
In the above, $B^m(\epsilon) = \{y\in \R^m : |y|<\epsilon\}$.
\end{lemma}
\begin{proof}
Fix $\epsilon>0$.
Set $\sN(t,x,y):= \sM(t,x,y)-\sM_0(t)$, so that $\sN(t,0,0)=0$.
Fix $r>0$ so small $\overline{B^m(r)}\subset W$.
Take $\gamma>0$ so small that if $x,|y|\leq \gamma$,
$\sup_{t\in [0,\epsilon_1]}\|\sN(t,x,y)\|\leq \frac{1}{2}\inf_{t\in [0,\epsilon_1]} \lambda_0(t).$
Without loss of generality, we assume $\epsilon<r\wedge \gamma$.
Let
$$C:= \sup_{\substack{t\in [0,\epsilon_1], x\in [0,\epsilon_0] \\ |y|\leq r}} |G(t,x,y)|<\infty.$$

Take $\delta\in (0,\gamma]$ so small that
\begin{equation*}
\frac{\epsilon}{\delta} \inf_{t\in [0,\epsilon_1]} \lambda_0(t)>2C, \quad \sup_{x\in [0,\delta]}|g_0(x)|<\epsilon.
\end{equation*}
Fix $x_0\in (0,\delta]$.
The Picard-Lindel\"of theorem shows that the solution $g^{x_0}(t)$ exists
and is unique
for $t$ in some interval $[0,s]$, where $s\in (0,\epsilon_1]$.
We will show that for $t\in [0,s]$, $|g^{x_0}(t)|<\epsilon$.  By iterating this process, it follows
that we do not have blow up in small time, and can take $s=\epsilon_1$.

Thus, we wish to show that for all $t\in [0,s]$, $|g^{x_0}(t)|<\epsilon$.
Suppose, for contradiction, there is $t_0\in [0,s]$ with $|g^{x_0}(t_0)|\geq \epsilon$.
Take the least such $t_0$.  Since $|g^{x_0}(0)|=|g_0(x_0)|<\epsilon$, $t_0>0$.
Hence, $|g^{x_0}(t_0)|=\epsilon$ and
\begin{equation}\label{EqnPerturbToContradictGrow}
\dt\bigg|_{t=t_0} |g^{x_0}(t)|^2 \geq 0.
\end{equation}

But, for $t\in [0,t_0]$, $|g^{x_0}(t)|\leq \epsilon< r\wedge \gamma$, and therefore,
\begin{equation*}
\begin{split}
&\dt |g^{x_0}(t)|^2 =
-\frac{1}{x_0}\ip{g^{x_0}(t)}{ \left(\sM_0(t)^{\top} +\sM_0(t)\right)g^{x_0}(t)}
\\&-\frac{1}{x_0} \ip{g^{x_0}(t)}{ \left(\sN(t,x_0,g^{x_0}(t))^{\top} +\sN(t,x_0,g^{x_0}(t))\right)g^{x_0}(t)}
+2\ip{g^{x_0}(t)}{ G(t,x_0,g^{x_0}(t))}
\\&\leq -\frac{1}{x_0} 2\lambda_0(t) |g^{x_0}(t)|^2 + 2\frac{1}{x_0}\|\sN(t,x_0,g^{x_0}(t))\| |g^{x_0}(t)|^2 + 2 |G(t,x_0,g^{x_0}(t)| |g^{x_0}(t)|
\\&\leq -\frac{1}{x_0} \lambda_0(t) |g^{x_0}(t)|^2 + 2C |g^{x_0}(t)|
\leq -\frac{1}{\delta} \lambda_0(t) |g^{x_0}(t)|^2 + 2C |g^{x_0}(t)|.
\end{split}
\end{equation*}
Hence,
\begin{equation*}
\dt\bigg|_{t=t_0} |g^{x_0}(t)|^2 \leq -\frac{\epsilon^2}{\delta}\lambda_0(t_0) +2C\epsilon
<0,
\end{equation*}
contradicting \cref{EqnPerturbToContradictGrow} and completing the proof.
\end{proof}

\begin{lemma}\label{LemmaPerturnSymmSoln}
Under the the assumption $\lambda_0(t)>0$ $\forall t$, there exists $\delta_0\in (0,\epsilon_0]$ and a function
$g(t,x)\in C([0,\epsilon_1]\times [0,\delta_0];W)$ such that
$g(0,x)=g_0(x)$ $\forall x\in [0,\delta_0]$, $g(t,0)=0$ $\forall t\in [0,\epsilon_1]$, and
$g$ satisfies \cref{EqnPerturbMainEqn}.
\end{lemma}
\begin{proof}
Let $\delta_0>0$ be the $\delta$ guaranteed by \cref{LemmaPerturbSymmFixed} with $\epsilon=1$.
For $x\in (0,\delta]$, set $g(t,x)=g^{x}(t)$, where $g^x(t)$ is the unique solution
from \cref{LemmaPerturbSymmFixed}.
Standard theorems from ODEs show $g(t,x):[0,\epsilon_1]\times (0,\delta_0]\rightarrow \R^m$ is continuous.  All that remains to show is that $g(t,x)$ extends to a continuous
function at $x=0$ by setting $g(t,0)=0$.  This follows immediately
from \cref{LemmaPerturbSymmFixed}.
\end{proof}

\begin{proof}[Proof of \cref{PropPertubMainProp}]
Set $\sMt(t,x,y):=-x\dot{R}(t)R(t)^{-1}  + R(t)\sM(t,x,R(t)^{-1}y)R(t)^{-1}$,
$\Gt(t,x,y)= R(t)G(t,x,R(t)^{-1}y)$.
Note that $\sMt(t,0,0)=\diag{\lambda_1(t),\ldots, \lambda_m(t)}$.
\cref{LemmaPerturnSymmSoln} shows that there exists $\delta_0\in (0,\epsilon_0]$
and  and a function
$h(t,x)\in C([0,\epsilon_1]\times [0,\delta_0];W)$ such that
$h(0,x)=R(0)g_0(x)$ $\forall x\in [0,\delta_0]$, $h(t,0)=0$ $\forall t\in [0,\epsilon_1]$, and
$h$ satisfies
\begin{equation*}
\dt h(t,w) = -\frac{1}{x} \sMt(t,x,h(t,x))h(t,x) + \Gt(t,x,h(t,x)), \quad x>0.
\end{equation*}
Setting $g(t,x)=R(t)^{-1} h(t,x)$ gives the desired solution, and completes the proof.
\end{proof} 

\section{Existence}
In this section, we prove \cref{ThmResultExist,ThmResultExistCharaterize}.
The key result needed for these, which is also useful for proving \cref{ThmResCharacterize}, is the next proposition. For it, we take all the same notation
 and assumptions as in the beginning of  \cref{SectionResults}.
 Thus, we have $m\in \N$, $\epsilon_0,\epsilon_1,\epsilon_2\in (0,\infty)$,
 $U\subseteq \R^m$ open, $D\in \N$,
 and $b_{\alpha,j}$, $c_{\alpha,j}$, $C_0$, $P$, and $\Ph$ as described in that section.

\begin{prop}\label{PropMainExistProp}
Let $A_0(t)\in C^2([0,\epsilon_1];U)$ and set $\sM(t):=-d_y P(t,0,A_0(t),A_0(t))$.
Suppose $\exists R(t)\in C^1([0,\epsilon_1];\GLm)$
with $R(t)\sM(t)R(t)^{-1}=\diag{\lambda_1(t),\ldots, \lambda_m(t)}$,
where $\lambda_j(t)>0$, for all $t,j$.
Take $c_0, C_1, C_2, C_3, C_4>0$ such that
$\min_{t,j} \lambda_j(t)\geq c_0>0$, $\CjN{1}{R}\leq C_1$, $\CjN{1}{R^{-1}}\leq C_2$,
$\CjN{1}{\sM^{-1}}\leq C_3$, $\CjN{2}{A_0}\leq C_4$.
Then, there exists
$\delta=\delta(m,D,c_0,C_0,C_1,C_2,C_3,C_4)>0$ and
$A(t,w)\in C^{0,2}([0,\epsilon_1]\times [0,\delta\wedge \epsilon_2];\R^m)$
such that
\begin{equation}\label{EqnExistMainPropAEqn}
\dt A(t,w) = \Ph(t,A(t,\cdot),A(t,0))(w), \quad A(t,0)=A_0(t).
\end{equation}
Moreover, if we set
\begin{equation}\label{EqnExistMainPropfEqn}
f(t,x)=
\begin{cases}
\frac{1}{x} \int_0^{\delta\wedge \epsilon_2} e^{-w/x} A(t,w)\: dw, &\text{if }x>0,\\
A_0(t), &\text{if }x=0,
\end{cases}
\end{equation}
then $f(t,x)\in C([0,\epsilon_1]\times [0,\epsilon_0];\R^m)$ and there exists
$\Gt(t,x)\in C([0,\epsilon_1]\times[0,\epsilon_0];\R^m)$ such that
\begin{equation}\label{EqnExistMainPropFdiffeq}
\dt f(t,x) = \frac{P(t,x,f(t,x),f(t,0)) - P(t,0,f(t,0),f(t,0))}{x} + \frac{1}{x^2} e^{-(\delta\wedge \epsilon_2)/x}\Gt(t,x), \quad f(t,0)=A_0(t).
\end{equation}
Finally, if $\delta_1\in [0,\epsilon_2\wedge \delta)$ and $\ft(t,x)\in C([0,\epsilon_1]\times [0,\epsilon_0];\R^m)$ satisfies
\begin{equation}\label{EqnExistMainPropFtdiffeq}
\dt \ft(t,x) = \frac{ P(t,x,\ft(t,x), \ft(t,0)) - P(t,0,\ft(t,0),\ft(t,0))}{x} + O(e^{-\delta_1/x}), \quad \ft(t,0) =A_0(t),
\end{equation}
then if $\lambda_0(t)=\min_{1\leq j\leq m} \lambda_j(t)$, we have $\forall \gamma\in [0,1)$,
\begin{equation*}
f(t,x) = \ft(t,x) + O\left(e^{-\delta_1/x} + e^{-\frac{\gamma}{x} \int_0^t \lambda_0(s)\: ds}\right).
\end{equation*}
In the above, the implicit constants in $O$ are independent of $(t,x)\in [0,\epsilon_1]\times [0,\epsilon_0]$.
\end{prop}

Without loss of generality, we may assume $\epsilon_2\leq 1$ in \cref{PropMainExistProp};
and we assume this for the rest of the section.  The heart of \cref{PropMainExistProp}
is an abstract existence result, which we now present.

\begin{prop}\label{PropExistAbsExist}
Fix $L\geq 0$.  Suppose $\sG:C^{0,L}([0,\epsilon_1]\times [0,\epsilon_2];\R^m)\rightarrow C^{0,L}([0,\epsilon_1]\times [0,\epsilon_2];\R^m)$ is an $(L, G_1, G_2, G_3)$ operation
(see \cref{DefnSmoothingOperatation}).  Let $\sM(t)\in C^{(L-1)\vee 0}([0,\epsilon_1];\M^{m\times m})$ be such that there exists $R(t)\in C^{1}([0,\epsilon_1];\GLm)$ satisfying $R(t)\sM(t)R(t)^{-1}=\diag{\lambda_1(t),\ldots, \lambda_m(t)}$, where $\lambda_j(t)>0$, $\forall j,t$.
Fix $A_0\in C^L([0,\epsilon_1];\R^m)$ and take
$c_0,C_1,C_2,C_3,C_4>0$ such that
$\min_{t,j} \lambda_j(t)\geq c_0>0$, $\CjN{1}{R}\leq C_1$, $\CjN{1}{R^{-1}}\leq C_2$,
$\CjN{(L-1)\vee 0}{\sM^{-1}}\leq C_3$, $\CjN{L}{A_0}\leq C_4$.
Then, there exists $\delta=\delta(L,m,G_1,G_2,G_3,c_0,C_1,C_2,C_3,C_4)>0$
such that there exists a solution $A(t,w)\in C^{0,L}([0,\epsilon_1]\times [0,\delta\wedge \epsilon_2];\R^m)$ to the equation
\begin{equation}\label{EqnExistAEqn}
\dt A(t,w) = -\sM(t) \dw A(t,w) + \sG(A)(t,w), \quad A(t,0)=A_0(t).
\end{equation}
\end{prop}

We prove \cref{PropExistAbsExist} by induction on $L$.  We begin with the inductive
step, which is contained in the next lemma.

\begin{lemma}\label{LemmaExistInductStep}
Let $L\geq 1$, and $\sG$, $A_0$, $\sM$, and $C_4$ be as in \cref{PropExistAbsExist}.
For $B(t,w)\in C^{0,L-1}([0,\epsilon_1]\times [0,\delta];\R^m)$
let $\sI(A_0,B)= A_0(t)+ \int_0^w B(t,r)\: dr$, and set
$\sGt_{A_0}(B)(t,w):= \dw \sG(\sI(A_0,B))(t,w)$, and let
$B_0(t)= \sM(t)^{-1} \left[-\dot{A}_0(t)+ \sG^{0}(A_0)(t)\right]\in C^{L-1}([0,\epsilon_1];\R^m)$ (here $\sG^0$ is as in \cref{DefnSmoothingOperatation}).
Then, $\sGt_{A_0}$ is an $(L-1, \Gt_1, \Gt_2, \Gt_3)$ operation, where
$\Gt_1$, $\Gt_3$, and $\Gt_3$ can be chosen to depend only on $G_1$, $G_2$,
$G_3$, and $C_4$.  Furthermore, consider the differential equation
\begin{equation}\label{EqnExistBEqn}
\dt B(t,w) = -\sM(t) \dw B(t,w) + \sGt_{A_0}(B)(t,w), \quad B(t,0)=B_0(t).
\end{equation}
Then, solutions to \cref{EqnExistAEqn} and \cref{EqnExistBEqn} are
in bijective correspondence in the following sense:
\begin{enumerate}[(i)]
\item\label{ItemExistASoln} If $A(t,w)\in C^{0,L}([0,\epsilon_1]\times [0,\delta];\R^m)$ is a solution to \cref{EqnExistAEqn}, then $B(t,w)=A'(t,w)\in C^{0,L-1}([0,\epsilon_1]\times [0,\delta];\R^m)$ is a solution to \cref{EqnExistBEqn}.

\item\label{ItemExistBSoln} If $B(t,w)\in C^{0,L-1}([0,\epsilon_1]\times [0,\delta];\R^m)$ is a solution to
\cref{EqnExistBEqn}, then $A(t,w) = \sI(A_0, B)(t,w)\in C^{0,L}([0,\epsilon_1]\times [0,\delta];\R^m)$ is a solution to \cref{EqnExistAEqn}.
\end{enumerate}
\end{lemma}
\begin{proof}
That $\sGt_{A_0}$ is an $(L-1, G_1,G_2,G_3)$ operation follows immediately
from the definitions.  Suppose $A(t,w)\in C^{0,L}([0,\epsilon_1]\times[0,\delta];\R^m)$
is a solution to \cref{EqnExistAEqn} and set $B(t,w) = A'(t,w)\in C^{0,L-1}([0,\epsilon_1]\times [0,\delta];\R^m)$.  Putting $w=0$
in \cref{EqnExistAEqn} and solving for $B(t,0)$ shows $B(t,0)=B_0(t)$.
Taking $\dw$ of \cref{EqnExistAEqn} and writing $A(t,w)=\sI(A_0,B)(t,w)$
shows $B$ satisfies $\dt B(t,w) = -\sM(t) \dw B(t,w) +\sGt_{A_0}(B)(t,w)$.
This proves (\ref{ItemExistASoln}).

Suppose $B(t,w)\in C^{0,L-1}([0,\epsilon_1]\times [0,\delta];\R^m)$ is a solution
to \cref{EqnExistBEqn} and set $A(t,w) = \sI(A_0,B)(t,w)\in C^{0,L}([0,\epsilon_1]\times[0,\delta];\R^m)$.  We wish to show
\cref{EqnExistAEqn} holds.  Clearly, $A(t,0)=A_0(t)$.  At $w=0$,
\cref{EqnExistAEqn} is equivalent to
$\dot{A}_0(t)+\sM(t) B_0(t)-\sG^{0}(A_0)(t)=0$, and this follows from the choice
of $B_0(t)$.  Thus, \cref{EqnExistAEqn} follows if:
\begin{equation}\label{EqnExistConstEqn}
\dw\left[ \dt A(t,w) + \sM(t) \dw A(t,w) - \sG(A)(t,w) \right]=0.
\end{equation}
But \cref{EqnExistConstEqn} is exactly \cref{EqnExistBEqn}, completing the proof.
\end{proof}

In light of \cref{LemmaExistInductStep}, it suffices to prove \cref{PropExistAbsExist}
in the case $L=0$.  The next lemma reduces this to the case when $\sM(t)$
is diagonal and $R(t)=I$.

\begin{lemma}\label{LemmaExistDiag}
Let $L=0$, and $\sG$, $A_0$, $\sM$, $\lambda_1,\ldots, \lambda_m$, and $R$
be as in \cref{PropExistAbsExist}.  For $B\in C([0,\epsilon_1]\times [0,\epsilon_2];\R^m)$,
set $\sGt(B)(t,w):= \dot{R}(t)R(t)^{-1} B(t,w) + R(t) \sG(R(\cdot)^{-1} B )(t,w)$.
Then, $\sGt$ is a $(0, \Gt_1, \Gt_2, \Gt_3)$ operation, where $\Gt_1$, $\Gt_2$,
and $\Gt_3$ can be chosen to depend only on $G_1$, $G_2$, $G_3$, $C_1$, and $C_2$.
Set $B_0(t):= R(t) A_0(t)$, and consider the differential equation
\begin{equation}\label{EqnExistDiagEqn}
\dt B(t,w) = -\diag{\lambda_1(t),\ldots, \lambda_m(t)} \dw B(t,w) + \sGt(B)(t,w), \quad B(t,0)=B_0(t).
\end{equation}
Then, solutions to \cref{EqnExistAEqn} and \cref{EqnExistDiagEqn}
are in bijective correspondence in the sense that $A(t,w)$ satisfies \cref{EqnExistAEqn}
if and only if $B(t,w)=R(t)A(t,w)$ satisfies \cref{EqnExistDiagEqn}.
\end{lemma}
\begin{proof}
This is immediate from the definitions.
\end{proof}

\begin{proof}[Proof of \cref{PropExistAbsExist}]
In light of \cref{LemmaExistInductStep,LemmaExistDiag} it suffices to
prove the result when $L=0$,
$\sM(t) = \diag{\lambda_1(t),\ldots, \lambda_m(t)}$.
Write $\sG(A)(t,w)=(\sG_1(A)(t,w),\ldots, \sG_m(A)(t,w))$,
then \cref{EqnExistAEqn} can be written as the system of differential equations
\begin{equation}\label{EqnExistDiagSystem}
\dt A_j(t,w) = -\lambda_j(t) \dw A_j(t,w)+ \sG_j(A)(t,w), \quad A(t,0)=A_0(t).
\end{equation}
Here, $A_0(t)\in C([0,\epsilon_1];\R^m)$, and the goal is to find a solution
$A(t,w)\in C([0,\epsilon_1]\times [0,\delta\wedge \epsilon_2];\R^m)$
to \cref{EqnExistDiagSystem} for some $\delta>0$.
The condition $A(t,0)=A_0(t)$ does not uniquely specify the solution to
\cref{EqnExistDiagSystem}.  We will prove the existence of a solution
to \cref{EqnExistDiagSystem} that, in addition, satisfies $A(0,w)=A_0(0)$.

Let $\delta>0$, to be chosen later, and set $\delta_0=\delta\wedge \epsilon_2$.
We consider $(t,w)$ in $[0,\epsilon_1]\times [0,\delta_0]$.
For each $j\in \mset$, let $V_j:=\dt+\lambda_j(t) \dw$.
For $u\geq 0$, let $\phi_{j,u}(v):=\int_u^{u+v} \lambda_j(r)\: dr$ and define
\begin{equation*}
\psi_{j,u}(r) :=
\begin{cases}
\phi_{j,u}^{-1}(r), &\text{if }\int_u^{\epsilon_1}\lambda_j(s)\: ds\geq r,\\
\epsilon_1, &\text{if }\int_u^{\epsilon_1} \lambda_j(s)\: ds\leq r.
\end{cases}
\end{equation*}
$V_j$ foliates $[0,\epsilon_1]\times [0,\delta_0]$ into integral curves.
We parameterize these integral curves by $u\in [-\delta_0,\epsilon_1]$:
when $u\leq 0$ we use the integral curve starting at $(0,-u)$ and
when $u\geq 0$ we use the integral curve starting at $(u,0)$.

More precisely, set
\begin{equation*}
U_{\epsilon_1,\delta_0}^j:=\{(u,v) : u\in [-\delta_0,\epsilon_1]\text{ and if }u\leq 0\text{ then }v\in [0,\psi_{j,0}(\delta_0+u)],\text{ and if }u\geq 0\text{ then }v\in [0,\psi_{j,u}(\delta_0)]\}.
\end{equation*}
Note, for $(u,v)\in U_{\epsilon_1,\delta_0}^j$, $v\leq \delta_0/c_0\leq \delta/c_0$.
Define $H_j:U_{\epsilon_1,\delta_0}^j\rightarrow [0,\epsilon_1]\times [0,\delta_0]$ by
\begin{equation*}
H_j(u,v) :=
\begin{cases}
(v, -u + \int_0^{v}\lambda_j(r)\: dr), &\text{if }u\leq 0,\\
(u+v, \int_u^{u+v} \lambda_j(r)\: dr), &\text{if }u\geq 0.
\end{cases}
\end{equation*}
Then, for each $u\in [-\delta_0,\epsilon_1]$, $H_j(u,\cdot)$ parameterizes
and integral curve of $V_j$:  when $u\leq 0$, it parameterizes
the curve starting at $(0,-u)$ and when $u\geq 0$, it parameterizes
the curve starting at $(u,0)$.  As such,
$H_j:U_{\epsilon_1,\delta_0}^j\rightarrow [0,\epsilon_1]\times [0,\delta_0]$
is a homeomorphism.

Define $L_0\in C([-\delta_0,\epsilon_1];\R^m)$ by $L_0(u)=A_0(u)$ for $u\geq 0$
and $L_0(u)=A_0(0)$ for $u\leq 0$.  We consider
$L=(L_1,\ldots, L_m)$ with $L_j(u,v)\in C(U_{\epsilon_1,\delta_0}^j)$.
We related $L$ and $A$ by the correspondence $L_j(u,v)=A_j\circ H_j(u,v)$.
We consider the system of differential equations
\begin{equation}\label{EqnExistLEqn}
\frac{\partial}{\partial v} L_j(u,v) = \sG_j(L_1\circ H_1^{-1},\ldots, L_m\circ H_m^{-1})(H_j(u,v)), \quad L_j(u,0)=L_{0,j}(u),
\end{equation}
where $L_0=(L_{0,1},\ldots, L_{0,m})$.
Note that if $L$ satisfies \cref{EqnExistLEqn}, then $A$ satisfies
\cref{EqnExistDiagSystem} and has $A(0,w)=A_0(0)$.
Thus, we complete the proof by finding $\delta>0$ such that there
is a solution to \cref{EqnExistLEqn}.  To do this, we utilize
the contraction mapping principle.

For $M>0$, let
\begin{equation*}
\sF_{M,\epsilon_1,\delta_0}:= \{ L=(L_1,\ldots, L_j) : L_j \in C(U_{\epsilon_1,\delta_0}^j), \CjN{0}{L_j}\leq M\},
\end{equation*}
and we give $\sF_{M,\epsilon_1,\delta_0}$ the metric
$\rho(L,\Lt)=\max_{1\leq j \leq m} \CjN{0}{L_j-\Lt_j}$, making $\sF_{M,\epsilon_1,\delta_0}$ into a complete metric space.

For $L\in \sF_{M,\epsilon_1,\delta_0}$, define
$\sT(L)=(\sT_1(L),\ldots, \sT_m(L))$, where $\sT_j(L)\in C(U_{\epsilon_1,\delta_0}^j)$
is defined by
\begin{equation*}
\sT_j(L)(u,v) := L_{0,j}(u)+\int_0^v \sG_j(L_1\circ H_1^{-1},\ldots, L_m\circ H_m^{-1})(H_j(u,v'))\: dv'.
\end{equation*}
We wish to pick $M$ and $\delta$ so that $\sT:\sF_{M,\epsilon_1,\delta_0}\rightarrow \sF_{M,\epsilon_1,\delta_0}$ is a strict contraction.
First, we pick $M$ and $\delta$ so that $\sT:\sF_{M,\epsilon_1,\delta_0}\rightarrow \sF_{M,\epsilon_1,\delta_0}$.  Indeed, we have
\begin{equation*}
|\sT_j(L)(u,v)|\leq \CjN{0}{A_0} + \int_0^v G_1(\sqrt{m}M)\: dr = \CjN{0}{A_0}+ v G_1(\sqrt{m}M) \leq C_4 + \frac{\delta}{c_0}G_1(\sqrt{m}M),
\end{equation*}
where in the last step we have used $v\leq \frac{\delta}{c_0}$, as noted earlier.
Set $M=2C_4$, then if $\delta\leq c_0C_4 G(\sqrt{m}M)^{-1}$, we have
$\sT:\sF_{M,\epsilon_1,\delta_0}\rightarrow \sF_{M,\epsilon_1,\delta_0}$.

We now wish to show that if we make $\delta$ sufficiently small, $\sT$ is a strict
contraction.  Consider, for $L,\Lt\in \sF_{M,\epsilon_1,\delta_0}$ we have
\begin{equation*}
|\sT_j(L)(u,v)-\sT_j(\Lt)(u,v)| \leq \int_0^v G_2(\sqrt{m}M)\rho(L,\Lt)\: dr \leq \frac{\delta}{c_0}G_2(\sqrt{m}M)\rho(L,\Lt),
\end{equation*}
where we have again used $v\leq \frac{\delta}{c_0}$.
Thus, $\rho(\sT(L), \sT(\Lt))\leq \frac{\delta}{c_0}G_2(\sqrt{m}M)\rho(L,\Lt)$.
Thus, if $\delta= \left( \frac{1}{2} c_0 G_2(\sqrt{m}M)^{-1} \right)\wedge \left( c_0 C_4 G_1(\sqrt{m}M)^{-1}\right)$,
$\sT:\sF_{M,\epsilon_1,\delta_0}\rightarrow \sF_{M,\epsilon_1,\delta_0}$ is a strict
contraction.

The contraction mapping principle applies to show that there is a fixed point
$L\in \sF_{M,\epsilon_1,\delta}$ with $\sT(L)=L$.
This $L$ is the desired solution to \cref{EqnExistLEqn}, which completes the proof.
\end{proof}

\begin{proof}[Proof of \cref{PropMainExistProp}]
We begin with the existence of $\delta>0$ and $A(t,w)\in C^{0,2}([0,\epsilon_1]\times [0,\delta\wedge \epsilon_2];\R^m)$  satisfying \cref{EqnExistMainPropAEqn}.
\cref{PropConvPolyIsOperation} shows that \cref{EqnExistMainPropAEqn}
is of the form covered by the case $L=2$ of \cref{PropExistAbsExist}.
Thus, the existence of $\delta$ and $A$ follow from \cref{PropExistAbsExist}.

Let $f$ be given by \cref{EqnExistMainPropfEqn}, so that for $x>0$,
\begin{equation*}
\dt f(t,x) = \frac{1}{x}\int_0^{\delta\wedge \epsilon_2} e^{-w/x} \Ph(t,A(t,\cdot),A(t,0))\: dw.
\end{equation*}
From here, \cref{EqnExistMainPropFdiffeq} follows from \cref{LemmaConvPolyDifferenceEqn}.

Finally, suppose $\ft$ is as in the statement of the proposition,
and set $g(t,x)=f(t,x)-\ft(t,x)$.  Since
$f(t,0)=\ft(t,0)=A_0(t)$, combining \cref{EqnExistMainPropFdiffeq,EqnExistMainPropFtdiffeq}
shows that there exists a bounded function
$\Gh(t,x):[0,\epsilon_1]\times (0,\epsilon_0]\rightarrow \R^m$
such that for $x\in (0,\epsilon_0]$,
\begin{equation}\label{EqnExistPropProofgEqn}
\dt g(t,x) = \frac{P(t,x,f(t,x),A_0(t))-P(t,x,\ft(t,x),A_0(t))}{x} + e^{-\delta_1/x} \Gh(t,x)
=-\frac{1}{x}\sM(t,x) + e^{-\delta_1/x} \Gh(t,x),
\end{equation}
where $\sM(t,x)=-\int_0^1 d_y P(t,x,sf(t,x)+(1-s) \ft(t,x), A_0(t))\: ds$.
In particular, note that $\sM(t,0)=\sM(t)$, since $f(t,0)=\ft(t,0)=A_0(t)$.
Solving \cref{EqnExistPropProofgEqn} we have
\begin{equation*}
g(t,x)= \Texp{-\frac{1}{x}\int_0^t \sM(s,x)\: ds}g(0,x) + e^{-\delta_1/x} \int_0^t \Texp{-\frac{1}{x} \int_s^t \sM(r,x)\: dr} \Gh(s,x)\: ds
\end{equation*}
Applying \cref{PropChronMainTexp}, we have $\forall \gamma\in[0,1)$,
\begin{equation*}
|g(t,x)|\lesssim e^{-\frac{\gamma}{x} \int_0^t \lambda_0(s)\: ds}|g(0,x)| + e^{-\delta_1/x} \int_0^{t} e^{-\frac{\gamma}{x}\int_s^{t} \lambda_0(r)\: dr}\: ds
=O\left(e^{-\delta_1/x} + e^{-\frac{\gamma}{x} \int_0^{t} \lambda_0(s)\: ds }\right),
\end{equation*}
completing the proof.
\end{proof}

\begin{proof}[Proof of \cref{ThmResultExist}]
Let $\fh(t,x)\in C([0,\epsilon_1]\times [0,\epsilon_0];\R^m)$ be the function
$f(t,x)$ from \cref{PropMainExistProp}.
Thus, $\fh$ satisfies \cref{EqnExistMainPropFdiffeq} for some function
$\Gt(t,x)\in C([0,\epsilon_1]\times [0,\epsilon_0];\R^m)$.

For some $\delta_0>0$, we will construct  $f(t,x)\in C([0,\epsilon_1]\times [0,\delta_0];\R^m)$
as in the statement of the theorem.  We do this by considering
$f(t,x)$ of the form $f(t,x)=\ft(t,x)+g(t,x)$, where $g(t,x)\in C([0,\epsilon_1]\times [0,\delta_0];\R^m)$.
Notice that $f(t,x)$ satisfies the conclusions of the theorem if $g(t,x)$
satisfies the following:
\begin{itemize}
\item $g(t,0)=0$, $\forall t\in [0,\epsilon_1]$ (so that $f(t,0)=\fh(t,0)=A_0(t)$).
\item $g(0,x)=g_0(x)$, where $g_0(x)=f_0(x)-\ft(0,x)\in C([0,\epsilon_0];\R^m)$.  Note, since $f_0(0)=A_0(0)=\ft(0,0)$, we have $g_0(0)=0$.
    \item 
    \begin{equation}\label{EqnExistGPerterbEqn}\dt g(t,x) = \frac{P(t,x,\ft(t,x)+g(t,x),A_0(t))-P(t,x,\ft(t,x),A_0(t))}{x} + G_1(t,x,g(t,x)),
    \end{equation}
        where $G_1(t,x,g(t,x)) = G(t,x,\ft(t,x)+g(t,x),A_0(t)) - \frac{1}{x^2}e^{-(\delta\wedge \epsilon_2)/x} \Gt(t,x)$ and $\delta$ is as in \cref{PropMainExistProp}.
\end{itemize}
Set
\begin{equation*}
\sMt(t,x,z) := -\int_0^1 (d_yP) (t,x, \ft(t,x)+sz,A_0(t))\: ds,
\end{equation*}
so that
\begin{equation*}
\sMt(t,x,z)z = P(t,x,\ft(t,x),A_0(t)) - P(t,x,\ft(t,x)+z,A_0(t)).
\end{equation*}
Using this, \cref{EqnExistGPerterbEqn} can be re-written as
\begin{equation*}
\dt g(t,x) = -\frac{1}{x}\sMt(t,x,g(t,x))g(t,x) + G_1(t,x,g(t,x)).
\end{equation*}
Also note, $\sMt(t,0,0)=\sM(t)$, where $\sM(t)$ is as in the statement of the theorem.
From here, the existence of $g(t,x)$ follows  from \cref{PropPertubMainProp},
completing the proof.
\end{proof}

\begin{proof}[Proof of \cref{ThmResultExistCharaterize}]
The representation \cref{EqnResulExistChar} follows by applying
\cref{PropMainExistProp} with $f$ playing the role $\ft$,
and $\delta_0$ playing the role of $\epsilon_0$.
The uniqueness of the representation follows from \cref{CorLaplaceAppendixUsed}.
\end{proof}

\section{Uniqueness}
The purpose of this section is to prove \cref{ThmResUnique,ThmResCharacterize,PropUniqueUniqueA,ThmResUniqueStability}. The main remaining ingredient needed is an abstract uniqueness result,
which we present first.

\subsection{An Abstract Uniqueness Result}

\begin{prop}\label{ThmAbsUnique}
Let $m\geq 1$, $\epsilon_1,\epsilon_2>0$.  Let $\sM(t)\in C([0,\epsilon_1];\M^{m\times m})$ be such that
there exists $R(t)\in C^1([0,\epsilon_1];\GLm)$ with $R(t)\sM(t)R(t)^{-1}=\diag{\lambda_1(t),\ldots, \lambda_m(t)}$ where each
$\lambda_j(t)>0$, $\forall t\in [0,\epsilon_1]$.  Suppose $g(t,w)\in C([0,\epsilon_1]\times [0,\epsilon_2];\R^m)$ satisfies
the differential equation
\begin{equation*}
\dt g(t,w) =\sM(t) \dw g(t,w) + F(t,w), \quad g(0,w)=0,\forall w,
\end{equation*}
where $F(t,w)$ satisfies $|F(t,w)|\leq C \sup_{0\leq r\leq w} |g(t,r)|$.
Set $\gamma_0(t):=\max_{1\leq j\leq m} \int_0^t \lambda_j(s)\: ds$,
and
\begin{equation*}
\delta_0 := \begin{cases}
\gamma_0^{-1}(\epsilon_2),&\text{if }\gamma_0(\epsilon_1)\geq \epsilon_2,\\
\epsilon_1,&\text{otherwise.}
\end{cases}
\end{equation*}
Then, $g(t,0)=0$ for $0\leq t\leq \delta_0$.
\end{prop}
\begin{proof}
We begin by showing that it suffices to prove the result in the case when $\sM(t)=\diag{\lambda_1(t),\ldots,\lambda_m(t)}$.
Indeed, if $g(t,w)$ is as above and $h(t,w)=R(t)g(t,w)$, then $h(t,w)$ satisfies
\begin{equation*}
\dt h(t,w) = \diag{\lambda_1(t),\ldots,\lambda_m(t)}\dw h(t,w) + R(t)F(t,w)+ \dot{R}(t)R(t)^{-1}h(t,w), \quad h(0,w)=0,\forall w.
\end{equation*}
Thus, if we have the result for $h$, the result for $g$ follows.

For the rest of the proof, we assume $\sM(t) = \diag{\lambda_1(t),\ldots,\lambda_m(t)}$.  Write $g(t,w)=(g_1(t,w),\ldots, g_m(t,w))$ and
$F(t,w)=(F_1(t,w),\ldots, F_m(t,w))$.  Thus we are interested in the system of equations
\begin{equation}\label{EqnUniqueSystemg}
\dt g_j(t,w) = \lambda_j(t) \dw g_j(t,w)+F_j(t,w), \quad g_j(0,w)=0,
\end{equation}
under the hypothesis $|F_j(t,w)|\leq C\sup_{0\leq r\leq w} |g(t,r)|$.
For each $j\in \mset$, set $\gamma_j(t)=\int_0^{t}\lambda_j(s)\: ds$, and let $Y_j=\dt-\lambda_j(t)\dw$.
Let $H_j(u,v) = \left(v, u-\int_0^v\lambda_j(s)\: ds\right)$ (we will be more precise about the domain of $H_j$ in a moment).
Note that $H_j$ is invertible with $H_j^{-1}(v,r)= \left(r+\int_0^v\lambda_j(s)\: ds, v\right)$.
Finally set
\begin{equation*}
\delta_j := \begin{cases}
\gamma_j^{-1}(\epsilon_2),&\text{if }\gamma_j(\epsilon_1)\geq \epsilon_2,\\
\epsilon_1,&\text{otherwise.}
\end{cases}
\end{equation*}

For $0\leq j\leq m$, set $W_j:=\{(t,w) : 0\leq t\leq \delta_j, 0\leq w\leq \epsilon_2-\gamma_j(t)\}$, and note that
for $j\in \mset$, $W_0\subseteq W_j$.  Furthermore, for $j\in \mset$, $Y_j$ foliates $W_j$ into the integral curves of $Y_j$.
Indeed, for $u\in [0,\epsilon_2]$, define
\begin{equation*}
r_j(u) := \begin{cases}
\gamma_j^{-1}(u),&\text{if }\gamma_j(\epsilon_1)\geq u,\\
\epsilon_1,&\text{otherwise.}
\end{cases}
\end{equation*}
Note $r_j(\epsilon_2)=\delta_j$.
As $v$ ranges from $0$ to $r_j(u)$, $H_j(u,v)$ parameterizes the integral curve of $Y_j$ in $W_j$ which starts at $(0,u)$.
Let $U_j:=\{(u,v) : u\in [0,\epsilon_2], v\in [0,r_j(u)]\}$.  By the above discussion, $H_j:U_j\rightarrow W_j$ is a homeomorphism.
Set $V_j:=H_j^{-1}(W_0)\subseteq U_j$.

For $v\in [0,\delta_0]$ define
\begin{equation*}
E(v):=\sup \{ |g(v,w)| : (v,w)\in W_0\} = \sup \{ |g(v,w)| : w\in [0,\epsilon_2-\gamma_0(v)]\}.
\end{equation*}
Clearly $E(0)=0$, since $g(0,w)=0$.  We will show $E(v)=0$ for $v\in [0,\delta_0]$, which will complete the proof.

We claim that if $(u,v)\in V_j$, then $\forall v'\in [0,v]$, $(u,v')\in V_j$.
Indeed, note that
\begin{equation*}
(u,v)\in V_j \Leftrightarrow v\in [0,\delta_0]\text{ and }0\leq u - \int_0^v\lambda_j(s)\: ds\leq \epsilon_2 - \max_k \int_0^v \lambda_k(s)\: ds.
\end{equation*}
So if $(u,v)\in V_j$ and $v'\in [0,v]$, then clearly $v'\in [0,\delta_0]$ and adding $\int_{v'}^{v} \lambda_j(s)\: ds$ to the above equation, we see
\begin{equation*}
0\leq \int_{v'}^v\lambda_j(s)\: ds \leq u - \int_0^{v'} \lambda_j(s)\: ds \leq \epsilon_2-\max_k\int_0^v\lambda_k(s)\: ds +\int_{v'}^v\lambda_j(s)\: ds \leq \epsilon_2 - \max_k \int_0^{v'}\lambda_k(s)\: ds.
\end{equation*}
Thus, $(u,v')\in V_j$, proving the claim.

Set $l_j(u,v)=g_j\circ H_j(u,v)$.  \cref{EqnUniqueSystemg} shows
\begin{equation*}
\frac{\partial}{\partial v} l_j(u,v) = F_j\circ H_j(u,v),\quad l_j(u,0)=g_j(0,u)=0.
\end{equation*}
Hence, $l_j(u,v) = \int_0^v F_j\circ H_j(u,v')\: dv'$.

For $(u,v)\in V_j$, $H_j(u,v)\in W_0$ and therefore $u-\int_0^v\lambda_j(s)\: ds\leq \epsilon_2-\gamma_0(v)$.
Hence, for $(u,v)\in V_j$,
\begin{equation*}
|F_j\circ H_j(u,v)|\lesssim \sup_{0\leq r\leq u-\int_0^{v} \lambda_j(s)\: ds} |g(v,r)|\leq \sup_{0\leq r\leq \epsilon_2-\gamma_0(v)} |g(v,r)| =E(v).
\end{equation*}
Thus, for $(u,v)\in V_j$, if $v'\in [0,v]$ we have $(u,v')\in V_j$ and therefore $|F_j\circ H_j(u,v')|\lesssim E(v')$.  We conclude, for $(u,v)\in V_j$,
\begin{equation*}
|l_j(u,v)| = \left| \int_0^v F_j\circ H_j(u,v')\: dv' \right| \lesssim \int_0^v E(v')\: dv'.
\end{equation*}

Therefore, for $v\in [0,\delta_0]$,
\begin{equation*}
\sup\{|g_j(v,w)| : (v,w)\in W_0 \} = \sup\{|l_j(u,v)|: (u,v)\in V_j\} \lesssim \int_0^v E(v')\: dv',
\end{equation*}
and so $E(v)\lesssim \int_0^{v} E(v')\: dv'$.  Gr\"onwall's inequality implies $E(v)=0$ for $v\in [0,\delta_0]$, completing the proof.
\end{proof} 

\subsection{Completion of the Proofs}
\begin{proof}[Proof of \cref{ThmResCharacterize}]
Set $\ft(t,x)=f(\epsilon_1-t,x)$, $\At_0(t)=f(\epsilon_1-t,0)=\ft(t,0)$, $\Pt(t,x,y,z)=-P(\epsilon_1-t,x,y,z)$.  $\ft$ satisfies, $\forall \gamma\in [0,\epsilon_2)$,
\begin{equation*}
\dt \ft(t,x) = \frac{\Pt(t,x,\ft(t,x), \ft(t,0))-\Pt(t,0,\ft(t,0),\ft(t,0))}{x}+O(e^{-\gamma/x}), \quad \ft(t,0)=\At_0(t).
\end{equation*}
By the hypotheses of the theorem, $\Pt$ and $\At_0$ satisfy all the hypotheses
of $P$ and $A_0$ in \cref{PropMainExistProp}.  Here,
$\lambdat_j(t)=\lambda_j(\epsilon_1-t)$ plays the role of $\lambda_j$
in that proposition.  Thus, let $\delta$ be as in \cref{PropMainExistProp}
and $\At\in C^{0,2}([0,\epsilon_1]\times [0,\delta];\R^m)$
be $A$ from \cref{PropMainExistProp} when applied to $\Pt$ and $\At_0$.
\Cref{PropMainExistProp} shows that $\forall \gamma\in [0,1)$,
if $\lambdat_0(t)=\min_{1\leq j\leq m} \lambdat_j(t)$,
\begin{equation}\label{EqnUniqueCompleteBeforeReverse}
\frac{1}{x} \int_0^{\delta\wedge \epsilon_2} e^{-w/x} \At(t,w)\: dw
= \ft(t,x) + O\left( e^{-\gamma(\epsilon_2\wedge \delta)/x} + e^{-\frac{\gamma}{x} \int_0^{t} \lambdat_0(s)\: ds} \right).
\end{equation}
Define $A(t,x):=\At(\epsilon_1-t,x)$.
Replacing $t$ with $\epsilon_1-t$ in \cref{EqnUniqueCompleteBeforeReverse}
and using that $\At$ satisfies \cref{EqnExistMainPropAEqn} (with $P$ and $A_0$
replaced by $\Pt$ and $\At_0$),
\cref{EqnResUniqueADiffEq,EqnResultCharacterIsLaplace} follow.
Finally, the stated uniqueness of \cref{EqnResultCharacterIsLaplace}
follows from \cref{CorLaplaceAppendixUsed}.
\end{proof}

\begin{proof}[Proof of \cref{PropUniqueUniqueA}]
Let $g(t,w)=A(t,w)-B(t,w)$.  \Cref{EqnResUniquePropUniqueEqn}
combined with \cref{PropConvPolyDifferencePh} shows
\begin{equation*}
\dt g(t,w) = \sM(t) \dw g(t,w)+F(t,w), \quad g(0,w)=0,
\end{equation*}
where $|F(t,w)|\lesssim \sup_{0\leq r\leq w} |g(t,r)|$.
$\sM(t)$ and $g(t,w)$ satisfy all the hypotheses of \cref{ThmAbsUnique}
(with $\epsilon_2$ replaced by $\delta'$), and the result follows
from \cref{ThmAbsUnique}.
\end{proof}

\begin{proof}[Proof of \cref{ThmResUniqueStability}]
Applying \cref{ThmResCharacterize} to $f_1$ and $f_2$
we see that there exists $\delta=\delta(m,D,c_0,C_0,C_1,C_2,C_3,C_4)>0$
and $A_1,A_2\in C^{0,2}([0,\epsilon_1]\times [0,\delta\wedge \epsilon_2];\R^m)$
such that for $k=1,2$, $\dt A_k(t,w)=\Ph(t,A_k(t,\cdot), A_k(t,0))(w)$, $A_k(t,0)=f_k(t,0)$,
and $\forall \gamma\in [0,1)$,
\begin{equation*}
f_k(0,x)=\frac{1}{x}\int_0^{\delta\wedge \epsilon_2} e^{-w/x} A_k(0,w)\: dw + O\left(  e^{-\gamma(\delta\wedge\epsilon_2)/x} + e^{-\frac{\gamma}{x} \int_0^{\epsilon_1} \lambda_0^k(s)\: ds  } \right).
\end{equation*}
The uniqueness of this representation as described in \cref{ThmResCharacterize},
combined with \cref{EqnUniqueStabilityInitial}, shows
that $A_1(0,w)=A_2(0,w)$ for $w\in [0,\delta'\wedge r]$.

From here, \cref{PropUniqueUniqueA} shows that
$A_1(t,0)=A_2(t,0)$ for $t\in [0,\delta_0]$.  Since $A_k(t,0)=f_k(t,0)$,
the result follows.
\end{proof}

\begin{proof}[Proof of \cref{ThmResUnique}]
This follows from the reconstruction procedure discussed in \cref{RmkResUniqueReconstruct}.
\end{proof}

\appendix
\section{The Laplace Transform}\label{AppendixLaplace}
The purpose of this section is to discuss the following Paley-Wiener type
theorem for the Laplace transform, which is contained in
\cite{SimonANewApproachToInverseSpectalTheoryI}.

\begin{thm}[Theorem A.2.2 of \cite{SimonANewApproachToInverseSpectalTheoryI}]\label{ThmAppendixSimon}
Fix $\epsilon>0$ and suppose $f,g\in L^1([0,\epsilon])$ and for some $s\in [0,\epsilon]$,
\begin{equation*}
\int_0^{\epsilon} e^{-\lambda t} f(t)\: dy = \int_0^\epsilon e^{-\lambda t} g(t)\: dy + O(e^{-s\lambda}),\text{ as } \lambda\uparrow\infty.
\end{equation*}
Then $f\equiv g$ on $[0,s)$.
\end{thm}

In this section, we offer a discussion of this result, along with two proofs.
The first is closely related to the proof in \cite{SimonANewApproachToInverseSpectalTheoryI}, though may be somewhat
simpler.  This first proof uses complex analysis.  The second proof uses
only real analysis and is more constructive.

\begin{lemma}\label{LemmaAppendixMySimon}
Fix $\epsilon>0$ and suppose $a\in L^1([0,\epsilon])$.
For each $\lambda\geq 1$, let $F(\lambda):= \int_0^\epsilon e^{-\lambda t} a(t)\: dt$.
Suppose $|F(\lambda)|= O(e^{-\epsilon\lambda})$ as $\lambda\uparrow \infty$.
Then, $a=0$.
\end{lemma}
\begin{proof}
For $\lambda \in \C$, set $G(\lambda)= \int_0^{\epsilon} e^{(\epsilon-t) \lambda} a(t)\: dt = e^{\epsilon\lambda} F(\lambda)$.
We have:
\begin{enumerate}[(a)]
\item $G$ is entire.
\item\label{ItemLinearPhragBounded1} $\sup_{\lambda\in \R} |G(i\lambda)|<\infty$.
\item\label{ItemLinearPhragBounded2} $\sup_{\lambda\in [0,\infty)} |G(\lambda)|<\infty$ (this is a restatement of the fact that $|F(\lambda)|= O(e^{-\epsilon\lambda})$).
\item\label{ItemLinearPhragBounded3} $\sup_{\lambda\in (-\infty,0]} |G(\lambda)|<\infty$.
\item\label{ItemLinearPhragLindeloffHyp} $|G(\lambda)|\leq C e^{\epsilon|\lambda|}$, for all $\lambda \in \C$.
\end{enumerate}
(\ref{ItemLinearPhragLindeloffHyp}) shows that we may apply the Phragm\'en-Lindel\"of principle in sectors
of angle less than $\pi$.
(\ref{ItemLinearPhragBounded1}), (\ref{ItemLinearPhragBounded2}), and (\ref{ItemLinearPhragBounded3}) show
$|G(\lambda)|$ is bounded on each coordinate axis, and  so the Phragm\'en-Lindel\"of principle shows that $G$ is bounded in each quadrant.  We conclude that $G$ is a bounded entire function
and therefore Liouville's theorem implies that $G$ is constant.  Since $\lim_{\lambda\rightarrow -\infty} G(\lambda)=0$, we see that $G(\lambda)=0$ for all $\lambda$.
Thus, $0=F(\lambda)=\int_0^\epsilon e^{-t\lambda} a(t)\: dt$ for all $\lambda$.  Standard theorems
now show $a=0$.
\end{proof}

\begin{proof}[Proof of \cref{ThmAppendixSimon}]
This follows immediately from \cref{LemmaAppendixMySimon}.
\end{proof}

In this paper, we use
\cref{ThmAppendixSimon,LemmaAppendixMySimon} via the next corollary.

\begin{cor}\label{CorLaplaceMainMyCor}
Suppose $a\in C([0,\epsilon])$ satisfies $|\lambda \int_0^{\epsilon} e^{-t\lambda} a(t)\: dt| = O(e^{-\epsilon \lambda})$, as $\lambda\uparrow \infty$.  Then, $a=0$.
\end{cor}
\begin{proof}This follows immediately from \cref{LemmaAppendixMySimon}.
\end{proof}

\begin{cor}\label{CorLaplaceAppendixUsed}
Let $\epsilon,\epsilon'>0$ and suppose $a,b\in C([0,\epsilon'])$ satisfy
\begin{equation*}
\frac{1}{x} \int_0^{\epsilon'} e^{-w/x} a(w)\: dw = \frac{1}{x}\int_0^{\epsilon'} e^{-w/x}b(w)\: dw + O(e^{-\epsilon/x})\text{ as }x\downarrow 0.
\end{equation*}
Then, $a(w)=b(w)$ for $w\in [0,\epsilon\wedge \epsilon']$.
\end{cor}
\begin{proof}
This follows from \cref{CorLaplaceMainMyCor} by setting $\lambda=\frac{1}{x}$.
\end{proof}

It is interesting to note that \cref{LemmaAppendixMySimon} (and therefore \cref{ThmAppendixSimon})
can be easily proved
without complex analysis, and we present this next.  Thus, all of the results in
this paper can be proved without complex analysis.  First, we note that
\cref{CorLaplaceMainMyCor} actually implies \cref{LemmaAppendixMySimon}.

\begin{proof}[Proof of \cref{LemmaAppendixMySimon} given \cref{CorLaplaceMainMyCor}]
Suppose $a\in L^1([0,\epsilon])$ and that $\int_0^{\epsilon} e^{-\lambda t} a(t)\: dt = O(e^{-\epsilon \lambda})$;
we wish to show $a=0$.  Integration by parts shows
\begin{equation*}
e^{-\lambda \epsilon} \int_0^{\epsilon} a(s)\: ds + \lambda\int_0^{\epsilon} e^{-\lambda t} \int_0^t a(s)\: ds\: dt = \int_0^{\epsilon} e^{-\lambda t} a(t) \: dt = O(e^{-\epsilon \lambda}).
\end{equation*}
Thus $\lambda \int_0^{\epsilon} e^{-\lambda t} \int_0^t a(s)\: ds\: dt = O(e^{-\epsilon \lambda})$,
and \cref{CorLaplaceMainMyCor} shows $\int_0^t a(s)\: ds=0$, $\forall t$.  Thus, $a=0$, as desired.
\end{proof}

Hence, to prove \cref{LemmaAppendixMySimon} using only real analysis, it suffices to prove 
\cref{CorLaplaceMainMyCor} using only real analysis, to which we now turn.

\begin{prop}\label{ThmQuantWeier}
Fix $\epsilon>0$, and let $a\in C([0,\epsilon])$.  Suppose
\begin{equation*}
\sup_{n\in \N} \left| n\int_0^{\epsilon} e^{nt} a(t)\: dt \right|<\infty.
\end{equation*}
Then, $a=0$.
\end{prop}

\begin{rmk}
Two remarks are on order:
\begin{itemize}
\item If $\int_0^{\epsilon} e^{nt} a(t)\: dt =0$, for all $n\in \N$, then the classical Weierstrass approximation easily yields that $a=0$.  It therefore makes sense to consider
\cref{ThmQuantWeier} a ``quantitative Weierstrass approximation theorem.''

\item By replacing $a(t)$ with $a(\epsilon-t)$, \cref{ThmQuantWeier} implies \cref{CorLaplaceMainMyCor}.
\end{itemize}
\end{rmk}

\begin{lemma}\label{LemmaQuantWeier}
Fix $\epsilon>0$, and let $a\in C([0,\epsilon])$.  Suppose
\begin{equation*}
\sup_{n\in \N} \left| n\int_0^{\epsilon} e^{nt} a(t)\: dt \right|<\infty.
\end{equation*}
Then, $a(0)=0$.
\end{lemma}

\begin{proof}[Proof of \cref{ThmQuantWeier} given \cref{LemmaQuantWeier}]
Let $\delta\in [0,\epsilon)$, and set $C=\sup_{n\in \N} \left| n\int_0^{\epsilon} e^{nt} a(t)\: dt \right|$.
Then, we have
\begin{equation*}
\left|n\int_\delta^{\epsilon} e^{nt} a(t)\: dt\right| \leq \left|n\int_0^\epsilon e^{nt} a(t)\: dt\right| + \left| n\int_0^\delta e^{nt} a(t)\: dt\right|
\leq C +  \left(\sup_{t\in [0,\delta]} |a(t)|\right) n\int_0^{\delta} e^{nt} \: dt\leq D e^{n\delta},
\end{equation*}
for some constant $D$ which does not depend on $n$.
Multiplying both sides of the above inequality by $e^{-n\delta}$ and applying the change of variables $s=t-\delta$, we have
\begin{equation*}
\left|n\int_0^{\epsilon-\delta} e^{ns} a(s+\delta)\: ds\right|\leq D, \quad \forall n\in \N.
\end{equation*}
\Cref{LemmaQuantWeier} now implies $a(\delta)=0$.  As $\delta\in [0,\epsilon)$ was arbitrary, this completes the proof.
\end{proof}

We close this appendix with a proof of \cref{LemmaQuantWeier}.  Fix $\epsilon>0$.  For $j,N\in \N$, define
\begin{equation*}
A_j := \int_1^\infty y^{j-1} e^{-y} \: dy, \quad I_{j,N}:=\int_1^{e^{\epsilon N}} y^{j-1} e^{-y}\: dy,
\end{equation*}
so that $A_j\leq A_{j+1}$ and $\lim_{N\rightarrow \infty} I_{j,N} = A_j$.  Set
\begin{equation*}
f_{j,N} (t):= \frac{N}{I_{j,N}} e^{Njt} e^{-e^{Nt}}.
\end{equation*}

\begin{lemma}\label{LemmaQuantWeierApproxI}
$f_{j,N}$ has the following properties.
\begin{itemize}
\item $\int_0^{\epsilon} f_{j,N}(t)\: dt=1$.
\item For $j$ fixed, $\lim_{N\rightarrow \infty} f_{j,N}(x) =0$ uniformly on compact subsets of $(0,\epsilon]$.
\item For $a\in C([0,\epsilon])$, $\lim_{N\rightarrow \infty} \int_0^\epsilon f_{j,N}(t) a(t)\: dt=a(0)$.
\end{itemize}
\end{lemma}
\begin{proof}
The last property follows from the first two.  The second property is immediate from the definitions.
We prove the first property.  Applying the change of variables $y=e^{Nt}$, we have
\begin{equation*}
\int_0^{\epsilon} f_{j,N}(t)\: dt = \frac{1}{I_{j,N}} \int_0^{e^{\epsilon N}} y^{j-1} e^{-y}\: dy = 1.
\end{equation*}
\end{proof}

\begin{proof}[Proof of \cref{LemmaQuantWeier}]
Let $a$ be as in the statement of the lemma, and set $C:=\sup_{n\in \N} \left| n\int_0^{\epsilon} e^{nt} a(t)\: dt \right|<\infty$.
Using \cref{LemmaQuantWeierApproxI}, we have
\begin{equation*}
\begin{split}
&|a(0)| = \lim_{N\rightarrow \infty} \left|\int_0^{\epsilon} f_{j,N}(t) a(t)\: dt\right|
\leq \liminf_{N\rightarrow \infty} \frac{N}{I_{j,N}} \sum_{k=0}^\infty \left|\int_0^\epsilon e^{Njt} \frac{\left(-e^{Nt}\right)^k }{k!}a(t)\: dt\right|
\\&\leq \liminf_{N\rightarrow \infty} \frac{N}{I_{j,N}} \sum_{k=0}^{\infty} \frac{C}{N(k+j) (k!)}
=\frac{1}{A_j} \sum_{k=0}^\infty \frac{C}{(k+j)(k!)}.
\end{split}
\end{equation*}
Taking the limit of the above equation as $j\rightarrow \infty$ shows $a(0)=0$, completing the proof.
\end{proof} 

\section{Pseudodifferential operators and the Calder\'on problem}\label{AppendixCalderon}
The results in this paper can serve as a model case for a more difficult (and still open) problem
involving pseudodifferential operators, which arises in the famous Calder\'on problem.

Let $N$ be a smooth manifold of dimension $n\geq 2$, and let $\PDO^s$ denote the space
of standard pseudodifferential operators on $N$ of order $s\in \R$.
We use $x$ to denote points in $N$.
For $T\in \PDO^s$, let $\sigma(T)$ denote the principal symbol of $T$.
Let $t\mapsto \Gamma(t)$ be a smooth
map $[0,\epsilon_1]\rightarrow \PDO^1$  such that $\Gamma(t)$ is elliptic for all $t$,
and such that:
\begin{equation*}
\sigma(\Gamma(t))(x,\xi) = \sqrt{|g(x,t)|\sum_{\alpha,\beta} g^{\alpha,\beta}(x,t) \xi_{\alpha}\xi_{\beta}},
\end{equation*}
where $g_{\alpha,\beta}(\cdot,t)$ is a Riemannian metric on $N$ for each $t\in [0,\epsilon_1]$, $|g(x,t)|$ denotes
$\mathrm{det}\: g_{\alpha,\beta}(x,t)$, and $\xi$ denotes the frequency variable.
In what follows, we suppress the dependance on $x$.
By taking principal symbols, the function $\Gamma(t)\mapsto |g(t)|g^{\alpha,\beta}(t)$
is well defined.  Also, $\det(|g|g^{\alpha,\beta})=|g|^{n-1}$, so (since $n\geq 2$),
$\Gamma(t)\mapsto |g(t)|$ is well-defined.  We conclude that
$\Gamma(t)\mapsto g_{\alpha,\beta}(t)$ is well defined.

Let $\lap_{g(t)}$ denote the Laplace-Beltrami operator
associated to $g(t)$ (with the convention that $\lap_{g(t)}$ is a negative operator).
We consider the following, well-known, differential equation:
\begin{equation}\label{EqnPDODiffEqn}
\dt \Gamma(t) = |g(t)|^{\frac{1}{2}} \left(|g(t)|^{-\frac{1}{2}} \Gamma(t)\right)^2 - \left(-|g(t)|^{\frac{1}{2}} \lap_{g(t)}\right).
\end{equation}
Notice, since $g(t)$ is a function of $\Gamma(t)$, \cref{EqnPDODiffEqn} can be considered
as a differential equation involving only $\Gamma(t)$.

\begin{conj}\label{ConjUniquePDO}
If $N$ is compact and without boundary, the differential equation \cref{EqnPDODiffEqn} has uniqueness.  I.e., if
$\Gamma_1(t)$ and $\Gamma_2(t)$ are as above
and both satisfy \cref{EqnPDODiffEqn} and $\Gamma_1(0)=\Gamma_2(0)$,
then $\Gamma_1(t)=\Gamma_2(t)$, $\forall t$.
\end{conj}

Note that the left hand side of \cref{EqnPDODiffEqn} is in $\PDO^1$,
while the right hand side is a difference of two elements of $\PDO^2$,
but this is possible since the principal symbols of the two terms on the right hand side cancel.   This makes this equation similar to the ones studied in this paper, as we discuss next.

\begin{rmk}
Other than this cancellation, as far as the methods in this paper are concerned,
there seems to be nothing particularly special about the form of \cref{EqnPDODiffEqn}
and one could state many other versions of \cref{ConjUniquePDO}
using different polynomials.  We will see in \cref{AppendixDefineCalderon}, and as is well-known,
\cref{EqnPDODiffEqn} arises naturally in the Calder\'on problem.
Thus, if one replaces \cref{EqnPDODiffEqn} with a more general polynomial
differential equation, one creates
a class of conjectures which ``generalize'' part of the Calder\'on problem.
These generalizations move beyond the setting where any ingredient in the problem is linear.
\end{rmk} 

\subsection{Translation invariant operators}\label{AppendixTransInvCalderon}
When $N=\R^n$, $n\geq 2$, if one replaces composition of pseudodifferential operators
with multiplication of their symbols, then \cref{EqnPDODiffEqn}
is of the form covered by our main theorems.  Another way of saying this is
that if the operators were all assumed to be translation invariant
on $\R^n$, then the equation \cref{EqnPDODiffEqn} is of the form
covered by our main theorems--and we describe this next.
Thus, \cref{ConjUniquePDO} can be viewed as a noncommutative
analog of \cref{ThmResUnique}.

Let $\Gamma(t)$ be as described in the previous section, satisfying \cref{EqnPDODiffEqn}
and assume that $\Gamma(t)$ is translation invariant.  Thus, $g(t)$ does not
depend on $x$ and $\Gamma(t)$ is given by a multiplier:
\begin{equation*}
\widehat{\Gamma(t) f} (\xi) = M(t,\xi) \hat{f}(\xi),
\end{equation*}
and $M$ satisfies the differential equation:
\begin{equation}\label{EqnForInvTransMEqn}
\dt M(t,\xi) = |g(t)|^{-\frac{1}{2}} M(t,\xi)^2 - |g(t)|^{\frac{1}{2}} \sum_{\alpha,\beta} g^{\alpha,\beta}(t)\xi_{\alpha} \xi_\beta,
\end{equation}
and satisfies
\begin{equation*}
M(t,\xi) = \sqrt{|g(t)|\sum_{\alpha,\beta} g^{\alpha,\beta}(t) \xi_\alpha\xi_\beta} + O(1),\text{ as }|\xi|\uparrow \infty.
\end{equation*}

For $1\leq \alpha\leq n$, let $e_\alpha$ denote the $\alpha$th standard basis element.
For a positive definite quadratic form
\begin{equation*}
B(\xi) =  |\gt| \sum_{\alpha,\beta} \gt^{\alpha,\beta} \xi_{\alpha}\xi_\beta,
\end{equation*}
where $\gt$ is a positive definite matrix, associate to $B$ the vector
$v$ indexed by $1\leq \alpha\leq \beta\leq n$
with $v_{\alpha,\beta}=\sqrt{B(e_{\alpha}+e_{\beta})}$.
Note that $v=(v_{\alpha,\beta})$ uniquely determines $\gt$, and therefore $B$,
and the function $\sF(v) := |\gt|^{-\frac{1}{2}}$ is well-defined and smooth (here we have used
$n\geq 2$ and argued as in the previous section).

For $1\leq \alpha\leq \beta\leq n$ and $x\geq 0$, define
\begin{equation*}
f_{\alpha,\beta}(t,x):=
\begin{cases}
xM\left(t, \frac{1}{x} (e_{\alpha}+e_{\beta})\right)&\text{if }x>0,\\
\sqrt{|g(t)|  \left(g^{\alpha,\alpha}(t) + 2 g^{\alpha,\beta}(t) + g^{\beta,\beta}(t)
\right)} &\text{if }x=0.
\end{cases}
\end{equation*}
Rewriting \cref{EqnForInvTransMEqn} in terms of $f_{\alpha,\beta}$ we see $f_{\alpha,\beta}$ satisfies the system of differential equations
\begin{equation}\label{EqnForInvTransFeqn}
\begin{split}
\dt f_{\alpha,\beta}(t,x) &= \frac{ |g(t)|^{-\frac{1}{2}} f_{\alpha,\beta}(t,x)^2 - |g(t)|^{-\frac{1}{2}} f_{\alpha,\beta}(t,0)^2 }{x}
\\& = \frac{ \sF( f(t,0)) f_{\alpha,\beta}(t,x)^2 - \sF(f(t,0)) f_{\alpha,\beta}(t,0)^2}{x}.
\end{split}
\end{equation}
Note that, by the assumption that $g(t)$ is positive definite, $f_{\alpha,\beta}(t,0)>0$, $\forall t$.
It follows that \cref{EqnForInvTransFeqn} is of the form covered by \cref{ThmResUnique},
where we have used the polynomial $P=(P_{\alpha,\beta})$, where
\begin{equation*}
P_{\alpha,\beta}(t,x,y,z) = \sF(z) y_{\alpha,\beta}^2.
\end{equation*}
Thus, under the restriction that $\Gamma(t)$ is translation invariant,
\cref{ConjUniquePDO} follows from \cref{ThmResUnique}.

\begin{rmk}
It is not difficult to simplify the above equation using Liouville transformations to
reduce the problem to considering, for instance, the case $P(t,x,y,z)=y^2$.
However, the generality of our approach lets us avoid such reductions.
\end{rmk} 

\subsection{The Calder\'on Problem}\label{AppendixDefineCalderon}
In this section, we describe how \cref{EqnPDODiffEqn} arises in the Calder\'on problem--which is well-known to experts.
Let $M$ be a smooth, compact Riemannian manifold with boundary of dimension $n+1\geq 3$.
Let $G$ denote the metric on $M$.  The Dirichet-to-Neumann map $\Lambda_G:C^\infty(\partial M)\rightarrow C^\infty(\partial M)$ is defined as follows.  Given $f\in C^\infty(\partial M)$, let $u\in C^\infty(M)$
be the unique solution to $\lap_G u=0$ on $M$, $u\big|_{\partial M}=f$.  $\Lambda_G$ is then
defined as $\Lambda_G f=\frac{\partial}{\partial \nu} f\big|_{\partial M}$, where $\nu$ denotes
the outward unit normal to $\partial M$.  The inverse problem is to construct $G$ given $\Lambda_G$.
There is one obvious obstruction:  if
$\Psi:M\rightarrow M$ is a diffeomorphism which fixes $\partial M$, then $\Lambda_G=\Lambda_{\Psi^{*} G}$ (where $\Psi^{*}G$ denotes the pull back of $G$ via $\Psi$).\footnote{This obstruction was noted by Luc Tartar.} Calder\'on's problem then asks if this is the only obstruction.
\newline\newline
\noindent\textbf{The Anisotropic Calder\'on Conjecture:}
Suppose $\Lambda_{G_1}=\Lambda_{G_2}$.  Then there is a diffeomorphism $\Psi:M\rightarrow M$, which
fixes the boundary, such that $G_1=\Psi^{*} G_2$.
\newline\newline
The above conjecture remains open, and has attracted a great deal of attention.
It began with work of Calder\'on \cite{CalderonOnAnInverse}.
When $M\subset \R^{n+1}$ and
in the so-called {\it isotropic} setting: $G_{i,j}(x) = c(x) \delta_{i,j}$, the problem is well understood
\cite{IsakovOnUniqueness,KohnVogeliusDeterminingI,KohnVogeliusDeterminingII,KohnVogeliusIdentification,SylvesterUhlmannAUniqunessTheorem,SylvesterUhlmannAGlobalUniquenessTheorem,SylvesterUhlmannInverseBoudaryValueProblems,NachmanSylvesterUhlmannAnNDimalBorg,NachmanReconstructionsFromBoundary,IsakovCompletenessOfProducts,AlessandriniStableDeterminationOfConductivity}.

Moving to the general (anisotropic) setting, much less is known.
In the real analytic category, the result is known in the affirmative
\cite{LeeUhlmannDeterminingAnisotropic,LassasUhlmannOnDetermining,LassasTaylorUhlmannTheDirichletToNeumannMap}.
In the smooth category, little progress has been made on the full anisotropic question.  In a big step forward, recent work of Dos Santos Ferreira, Kenig,
Salo, and Uhlmann  \cite{DSFKenigSaloUhlmannLimitingCarlemanWeights,KenigSaloUhlmannReconstructionsFromBoundaryMesurements}
have given some of the first results in this setting.  However, they still require a special form of the metric $G$, and even then do not answer the full Calder\'on question.

\begin{rmk}
When $n+1=2$, the problem takes a slightly different form, and is very well understood
\cite{NachmanGlobalUniquenessForATwoDiml,SylvesterAnAnisotropicInverseBoundary,SunUhlmannAnisotropicInverseProblemsInTwoDimension,BrownUhlmannUniquenessInTheInverseConductivity,AstalaPaivarintaCalderonsInverseConductivityProblem,AstalaLassasCalderonInverseProblemForAnisotropic}.
Because of this, our main interest is the case $n+1\geq 3$.
\end{rmk}

Following \cite{LeeUhlmannDeterminingAnisotropic}, we use boundary normal coordinates on a neighborhood
of $\partial M$.  This sees a neighborhood of $\partial M$ in the form $\partial M\times [0,\epsilon)$.
We use coordinates $(x,t)\in \partial M\times [0,\epsilon)$.
$M$ has dimension $n+1$ and $\partial M$ has dimension $n$.  In what follows, $\alpha,\beta$
range over the numbers $1,\ldots, n$ while $i,j$ index the numbers $1,\ldots, n+1$.
In boundary normal coordinates, $G_{i,j}$ satisfies $G_{n+1,n+1}=1$, $G_{n+1,\beta}=0$, $G_{\alpha,n+1}=0$.
Let $g_{\alpha,\beta}(x,t)=G_{\alpha,\beta}(x,t)$; in particular, $g_{\alpha,\beta}(x,t)$ is an $n\times n$ matrix
and satisfies $\det g_{\alpha,\beta}(x,t) = \det G_{i,j}(x,t)$.  

For each $t_0\in [0,\epsilon)$, we shrink the manifold $M$ but cutting off the part of the manifold $[0,t_0)\times \partial M$ (in boundary normal coordinates), yielding a new Riemannian manifold $M_{t_0}$.   Let
$G_{t_0}$ denote the metric on $M_{t_0}$ (given by restricting $G$ to $M_{t_0}$).
For each $t_0\in[0,\epsilon)$, we think of $g(x,t_0)$ as a metric on $\partial M\cong \partial M_{t_0}$ (where we identify $\partial M$ with $\partial M_{t_0}$ in the obvious way).  We sometimes suppress
the variable $x$ and write $g(t_0)$ to denote the metric, which depends smoothly on $t_0$.

 For each $t_0$
we define the map $\Gamma(t_0):C^{\infty}(\partial M)\rightarrow C^{\infty}(\partial M)$ as follows.
Let $u_{t_0}$ solve $\lap_{G_{t_0}} u_{t_0} =0$ in $M_{t_0}$ with $u_{t_0}\big|_{\partial M_{t_0}}=f$ (here we are again identifying
$\partial M_{t_0}$ with $\partial M$ in the obvious way).   Then define
\begin{equation}\label{EqnCalderonDEDefinGamma}
\Gamma(t_0) f(x):= -|g(x,t)|^{\frac{1}{2}} \frac{\partial}{\partial t}\bigg|_{t=t_0} u_{t_0}(t,x).
\end{equation}
Note that $\Gamma(0)=|g(0)|^{\frac{1}{2}}\Lambda_G$.   Because it is well-known that
$\Lambda_G$ uniquely determines $G$ on $\partial M$, the Calder\'on problem can be equivalently
stated with $\Lambda_G$ replaced by $\Gamma(0)$.

We have
$$\lap_G  = \lap_{g(t)}+ |g(x,t)|^{-\frac{1}{2}}\frac{\partial}{\partial t} |g(x,t)|^{\frac{1}{2}} \frac{\partial}{\partial t}.$$
Differentiating \cref{EqnCalderonDEDefinGamma} with respect to $t$, using the above formula
for $\lap_G$, and using $\lap_{G_{t_0}} u_{t_0} =0$, we see
that $\Gamma(t)$ satisfies the differential equation \cref{EqnPDODiffEqn}.

Hence, if \cref{ConjUniquePDO} were true, it would follow that $\Gamma(0)$ uniquely
determines $g(t)$.  I.e., that $\Lambda_G$ uniquely determines $G$ on a neighborhood
of the boundary in boundary normal coordinates.

\begin{rmk}
In the real analytic category, differential equations always have uniqueness,
and the above argument shows that, for a real analytic manifold, $\Lambda_G$ uniquely 
determines $G$ on a neighborhood of the boundary, in boundary normal coordinates.
This is equivalent to the first step of \cite{LeeUhlmannDeterminingAnisotropic}, where the same
ideas are used to determine the Taylor series of $g$ in the $t$-variable, centered at $t=0$.
\end{rmk}


\bibliographystyle{amsalpha}

\bibliography{singular}

\center{\it{University of Wisconsin-Madison, Department of Mathematics, 480 Lincoln Dr., Madison, WI, 53706}}

\center{\it{street@math.wisc.edu}}

\end{document}